\documentclass[11pt]{amsart}%
\usepackage{graphicx}
\usepackage{color}
\usepackage{amsmath}
\usepackage{amsfonts}
\usepackage{amssymb}
\providecommand{\U}[1]{\protect\rule{.1in}{.1in}}

\providecommand{\U}[1]{\protect\rule{.1in}{.1in}}
\providecommand{\U}[1]{\protect\rule{.1in}{.1in}}
\providecommand{\U}[1]{\protect\rule{.1in}{.1in}}
\providecommand{\U}[1]{\protect\rule{.1in}{.1in}}
\newtheorem{theorem}{Theorem}[section]

\newtheorem{conjecture}[theorem]{Conjecture}
\newtheorem{corollary}[theorem]{Corollary}

\newtheorem{definition}[theorem]{Definition}

\newtheorem{lemma}[theorem]{Lemma}

\newtheorem{proposition}[theorem]{Proposition}
\newtheorem{remark}[theorem]{Remark}

\DeclareMathOperator{\Mult}{Mult}
\DeclareMathOperator{\Id}{Id}
\DeclareMathOperator{\Conv}{Conv}
\begin{document}
\title[Central measures on multiplicative graphs]{Central measures on multiplicative graphs, representations of Lie algebras and
weight polytopes}
\author{C\'{e}dric Lecouvey and Pierre Tarrago}
\date{December 2016}
\maketitle

\begin{abstract}
To each finite-dimensional representation of a simple Lie algebra is
associated a multiplicative graph in the sense of Kerov and Vershik defined
from the decomposition of its tensor powers into irreducible components. It
was shown in \cite{LLP1} and \cite{LLP2} that the conditioning of natural
random Littelmann paths to stay in their corresponding Weyl chamber is
controlled by central measures on this type of graphs. Using the K-theory of associated $C^{*}$-algebras, Handelman \cite{Han1} established a homeomorphism between the set of central measures on these multiplicative graphs and the weight polytope of the
underlying representation. In the present paper, we make explicit this homeomorphism independently of Handelman's results by using Littelmann's path model. As a by-product we also get an explicit parametrization of the
weight polytope in terms of drifts of random Littelmann paths. This explicit parametrization yields a complete description of harmonic and $c$-harmonic functions for this Littelmann paths model.

\end{abstract}

\section{Introduction}

Consider a simple Lie algebra $\mathfrak{g}$ of rank $d$ over $\mathbb{C}$ and
its root system in $\mathbb{R}^{d}$. Let $P$ be the corresponding weight
lattice and fix $\Delta$ a dominant Weyl chamber. Then $P_{+}=P\cap\Delta$ is
the cone of dominant weights of $\mathfrak{g}$. Denote by $S=\{\alpha
_{1},\ldots,\alpha_{d}\}$ the underlying set of simple roots.\ To each dominant
weight $\delta\in P_{+}$ corresponds a finite-dimensional representation
$V(\delta)$ of $\mathfrak{g}$ of highest weight $\delta$. In \cite{Lit1}
Littelmann associated with $V(\delta)$ a set $B(\delta)$ of paths in
$\mathbb{R}^{d}$ with length $1$ starting at $0$ with ends the set
$\Pi_{\delta}$ of weights of $V(\delta)$. Random Littelmann paths can then be
defined first by endowing $B(\delta)$ with a suitable probability
distribution, next by considering random concatenations of paths in
$B(\delta)$. In \cite{LLP1} and \cite{LLP2} distributions on the set
$B(\delta)$ are defined from morphisms from $P$ to $\mathbb{R}_{>0}$. This is
equivalent to associate to each simple root $\alpha_{i}$ a real $t_{i}$ in
$]0,+\infty\lbrack$. It is then shown that these random paths and their
conditioning to stay in the Weyl chamber $\Delta$ are controlled by the
representation theory of $\mathfrak{g}$. In fact, one so obtains particular
central distributions on the set $\Gamma_{n}(\mathbb{R}^{d})$ of paths of
any length $n\geq1$ (obtained by concatenating $n$ paths in $B(\delta)$).\ By
central distributions we here mean that the probability of a finite path only
depends on its length and its end. Equivalently, we get a central measure on
the set of infinite concatenations $\Gamma(\mathbb{R}^{d})$ of paths in
$B(\delta)$ (see Section \ref{Sec_generalities}). 

Write $\mathcal{H}_{\infty}(\mathbb{R}^{d})$ for the set of central measures
on $\Gamma(\mathbb{R}^{d})$ and $\mathcal{H}_{\infty}(\Delta)$ for the subset
of $\mathcal{H}_{\infty}(\mathbb{R}^{d})$ of central measures on
$\Gamma(\Delta)$, the set of infinite trajectories remaining in $\Delta$. By
Choquet Theorem both sets $\mathcal{H}_{\infty}(\mathbb{R}^{d})$ and
$\mathcal{H}_{\infty}(\Delta)$ are simplices so they are essentially
determined by their minimal boundaries $\partial\mathcal{H}_{\infty}%
(\mathbb{R}^{d})$ and $\partial\mathcal{H}_{\infty}(\Delta)$.\ Write
$K(\delta)$ for the convex hull of $\Pi_{\delta}$ and set $K(\delta
)^{+}=\Delta\cap K(\delta)$. For walks in the Weyl chambers, the characterization of the sets $\partial\mathcal{H}_{\infty}%
(\mathbb{R}^{d})$ and $\partial\mathcal{H}_{\infty}(\Delta)$ has been obtained by Handelman in \cite{Han1} and \cite{Han2} using 
a deepful work of Price \cite{Pr1,Pr2}, by proving that they are respectively homeomorphic to $K(\delta)$ and $K(\delta)^{+}$. Nevertheless, the relevant homeomorphisms are not here made explicit. 
Their existence is established by considering the central measures as traces on certain fixed point $C^{*}$-algebras, and then using analytic tools. 
In particular, a central element of the proof is the extension of traces on $C^{*}$-algebras using K-theory (a short explanation of these arguments is given in Section \ref{explanationProof}). 

The goal of this paper is essentially threefold: first we make explicit
both homeomorphisms by using the Weyl characters of $\mathfrak{g}$ (see Theorem
\ref{mainresult}), next we give a purely algebraic self containing alternative proof of Handelman's results and finally we connect them with 
more recent studies on conditioned random walks or Brownian motions, Pitman transform generalizations and asymptotic Young tableaux (see \cite{OC1}, \cite{BBO}, \cite{LLP1,LLP2,LLP3}, \cite{Chh}, \cite{De}, \cite{SN}).  As a corollary of these results, we describe the set of harmonic and $c$-harmonic functions corresponding to the aforementioned random walks. Finally, we get a law of large numbers for random walks following the central measures obtained. Our two last results seem quite disconnected from the initial algebraic setting in representation theory, and we conjecture that they still hold for a very broad class of random paths. Our approach extends that of Kerov and Vershik to which
it essentially reduces when $V(\delta)$ is the defining representation of
$\mathfrak{g}=\mathfrak{sl}_{n}$. Nevertheless, numerous difficulties arise
when considering the general case of dominant weights of any simple algebra
$\mathfrak{g}$, which explains the involved material used in the proof of Handelman. Our methods to determine $\partial\mathcal{H}_{\infty}%
(\mathbb{R}^{d})$ and $\partial\mathcal{H}_{\infty}(\Delta)$ are quite
similar. So we will now give its main steps only in the case of $\partial\mathcal{H}%
_{\infty}(\Delta)$.

We first need to show that the characterization of $\partial\mathcal{H}_{\infty}%
(\Delta)$ is equivalent to that of the extremal harmonic
functions on the growth graph $\mathcal{G}(\Delta)$
associated with $\Gamma(\Delta)$. This growth diagram is rooted, graded and
multiplicative: its vertices label the basis $\mathcal{B}=\{(s_{\lambda
},n)\mid V(\lambda)$ irreducible component of $V(\delta)^{\otimes n}$ and
$n\geq1\}$ of a commutative algebra $\hat{T}_{\delta}^{+}$ (here
$s_{\lambda}$ is the Weyl character of $V(\lambda)$). We then establish that
the extremal nonnegative harmonic functions on $\mathcal{G}(\Delta)$ are in bijection with the algebra morphisms from $\hat{T}_{\delta}^{+}$ to $\mathbb{R}$ that are nonnegative on $\mathcal{B}$. Next, we
prove that all these morphisms are obtained by associating to each simple root
$\alpha_{i},i=1,\ldots,n$ a real in $[0,1]$. The difficulty here comes from
the fact that two such associations can yield the same morphism. So to obtain
a genuine parametrization we need to restrict ourselves to a subset
$[0,1]_{\delta}^{d}$ (see (\ref{def[0,1]delta}) for a precise definition) of $[0,1]^{d}$ whose
combinatorial description is in terms of the $\delta$-admissible subsets of $S$
introduced in \cite{Vin}. Finally, in Proposition \ref{bijectionPolytop}, we
show that our set $[0,1]_{\delta}^{d}$ also parametrizes the simplex
$K(\delta)^{+}$ by considering, for each $d$-tuple in $[0,1]_{\delta}^{d}$, the
drift of the corresponding random Littelmann path appearing in the
construction of \cite{LLP1} and \cite{LLP2}.

\bigskip

The paper is organized as follows. In Section 2, we recall some background on
random chains, central measures and multiplicative graphs. We also give a
generalization of a Theorem by Kerov and Vershik relating extremal harmonic
functions on a multiplicative graph to nonnegative morphisms of the underlying
algebra. The main result is written down in Section 3 where we also introduce
the algebras $\hat{T}_{\delta}$ and $\hat{T}_{\delta}^{+}$; a sketch of Handelman's arguments is proposed at the end of Section 3.\ Section 4 gives the description of
$\partial\mathcal{H}_{\infty}(\mathbb{R}^{d})$. Here, we define our set
$[0,1]_{\delta}^{d}$ and relate it to the geometry of the polytope $K(\delta)$.
The description of $\partial\mathcal{H}_{\infty}(\Delta)$ is deduced from that
of $\partial\mathcal{H}_{\infty}(\mathbb{R}^{d})$ in Section 5. It is worth
noticing that we need here (as in the result of Kerov and Vershik) a classical
theorem relating polynomials with non positive roots to totally positive
sequences.\ Another important ingredient in the proof is the use of certain
plethyms of Schur and Weyl characters of $\mathfrak{g}$. Finally, Section 6
relates both descriptions of $\partial\mathcal{H}_{\infty}(\mathbb{R}^{d})$ and
$\partial\mathcal{H}_{\infty}(\Delta)$ to the drift of random Littelmann paths.
Notably it explains how the polytope $K(\delta)$ can be simply parametrized by
using the set $[0,1]_{\delta}^{d}$.

\section{General probabilistic framework}

\label{Sec_generalities}We present here a general probabilistic model of
random paths in a domain, which is well suited to study probabilistic aspects of
Littelmann paths and their asymptotics. We introduce first a discrete version of
paths in a vector space.

\subsection{Random paths on a lattice}

Let $d\geq0$ and let $\Lambda$ be a lattice of $\mathbb{R}^{d}$.

\begin{definition}
Let $n$ be a nonnegative integer. A path $\gamma$ on $\Lambda$ is a piecewise
linear function $\gamma:[0,n]\longrightarrow\mathbb{R}^{d}$ with $\gamma(0)=0$, $\gamma
(i)\in\Lambda$ for all $i\in\{0,...,n\}$, and $\gamma
(x)\in\Lambda$ for all $x$ for which $\gamma$ is not differentiable at $x$.
The path $\gamma$ is called infinitesimal if $t=1$ and $\gamma(0)=0$. 
The length of the path $\gamma$ is defined as the length of the interval on
which $\gamma$ is defined and denoted by $l(\gamma)$: the path is said finite
if its length is finite, and infinite otherwise.
\end{definition}

We denote by $\gamma.\tau$ the concatenation of two finite paths $\gamma$ and $\tau$. A path defines a sequence of vectors $(\gamma(0),\dots,\gamma(i),\dots)$ in
$\Lambda$. Let $k\in\mathbb{N}$. When $\gamma$ is a path of length
$n\geq k$, we denote by $\gamma_{\downarrow k}$ the path $\gamma_{|[0,k]}$.
Let $X$ be a denumerable set of infinitesimal paths and let $\Omega$ be a
domain of $\mathbb{R}^{d}$ such that $0\in\Omega$; from now on, the set $X$ is fixed and is not mentioned in the various notations. A path $\gamma$ is called $X$-valued if $\gamma$
is the concatenation of infinitesimal paths coming from $X$: equivalently,
$\left(  \gamma_{\vert[i,i+1]}-\gamma(i)\right)  \in X$ for all $i\geq1$. In the sequel, any path is always considered as $X$-valued. The
set of $X$-valued paths (resp. finite $X$-valued paths, resp. $X$-valued paths
of length $n$, with $n\in\mathbb{N}\cup\{\infty\}$) whose image is included in
$\Omega$ is denoted by $\Gamma(\Omega)$ (resp. by $\Gamma_{f}(\Omega)$, resp.
$\Gamma_{n}(\Omega)$). For $x,y\in\Lambda$, we denote by
$\Gamma_{\Omega}(x,y)$ the set of infinitesimal paths $\gamma\in X$ such that $\gamma(1)=y-x$ and $x+\gamma\subset \Omega$. Finally, we denote by $\Gamma_{\Omega}(y,n)$ the set of finite paths of length $n$ ending at $y$.

In order to consider random paths in $\Omega$, we
need to define a $\sigma$-algebra on $\Gamma(\Omega)$. Let $\tau$ be a finite
rooted path of length $n$, and let $\Gamma_{\Omega}(\tau)$ be the set
$\{\gamma\in\Gamma(\Omega)|l(\gamma)\geq n,\gamma_{\downarrow n}=\tau\}$. We
define the $\sigma$-algebra $\mathcal{A}$ as the coarsest $\sigma$-algebra
containing all the sets $\Gamma_{\Omega}(\tau)$ for $\tau\in\Gamma_{f}%
(\Omega)$. It is readily seen that $\Gamma_{f}(\Omega)\in\mathcal{A}$ and that
the restriction of $\mathcal{A}$ to $\Gamma_{f}(\Omega)$ is the discrete
$\sigma$-algebra. The set $M_{1}(\Gamma(\Omega))$ of probability measures on
$\Gamma(\Omega)$ is considered with the initial topology with respect to the
evaluation maps on the sets $\Gamma_{\Omega}(\tau),\tau\in\Gamma_{f}(\Omega)$.
By Tychonov's Theorem, $M_{1}(\Gamma(\Omega))$ is a compact set with respect
to this topology.

\subsection{Central random paths}

\begin{definition}
\label{defHarmRandompath} A random path $\omega$ in $\Gamma(\Omega)$ is called
central if there is a function $p:\Lambda\times\Lambda\times\mathbb{N}%
\longrightarrow\mathbb{R}^{+}$ such that
\[
\mathbb{P}(\omega\in\Gamma_{\Omega}(\gamma))=p(\gamma(0),\gamma(l(\gamma
)),l(\gamma)),
\]
for all $\gamma\in\Gamma_{f}(\Omega)$. A measure on $\Gamma(\Omega)$ is called
central if the corresponding random path is central.
\end{definition}

The set of central measures (resp. central measures supported on $\Gamma
_{f}(\Omega)$, resp. central measures supported on $\Gamma_{\infty}(\Omega)$)
is denoted by $\mathcal{H}(\Omega)$ (resp. $\mathcal{H}_{f}(\Omega),\mathcal{H}_{\infty
}(\Omega)$). The sets $\mathcal{H}(\Omega)$, $\mathcal{H}_{f}(\Omega)$ and $\mathcal{H}_{\infty
}(\Omega)$ are convex subsets of $M_{1}(\Gamma(\Omega))$. Conditioning elements of
$\mathcal{H}(\Omega)$ on $\Gamma_{f}(\Omega)$ and $\Gamma_{\infty}(\Omega)$ yields
that any central measure is a convex combination of central measures in
$\mathcal{H}_{f}(\Omega)$ and $\mathcal{H}_{\infty}(\Omega)$. Therefore, the description
of $\mathcal{H}(\Omega)$ is equivalent to the description of $\mathcal{H}_{f}(\Omega)$
and $\mathcal{H}_{\infty}(\Omega)$.

It is readily seen that there is an
alternative equivalent definition of central random paths: a random path $\omega$ is
central if and only if the law of $\omega_{\downarrow n}$ conditioned on the set
$\{\gamma\in\Gamma(\Omega)|l(\gamma)\geq n,\gamma(n)=y\}$ is the
uniform law on $\Gamma_{\Omega}(y,n)$. 
This equivalent definition gives a straightforward description of the set
$\mathcal{H}_{f}(\Omega)$. Namely, conditioning on the last point of the
random path yields that any central measure $P\in\mathcal{H}_{f}(\Omega)$ admits a
unique decomposition
\[
P=\sum_{\substack{y\in\Lambda\\n\geq1}}a_{y,n}P_{y,n},
\]
where $a_{y,n}\geq0$ and $P_{y,n}$ is the uniform distribution on
the set $\Gamma_{\Omega}(y,n)$ for $y\in\Lambda$ and $n\geq1$. On the other
hand, the description of the set $\mathcal{H}_{\infty}(\Omega)$ is much more
complicated. It is known (see the next
section) that $\mathcal{H}_{\infty}(\Omega)$ is a convex set and even a Choquet
simplex. Therefore, there exists a subset $\partial\mathcal{H}_{\infty
}(\Omega)\subset\mathcal{H}_{\infty}(\Omega)$, such that any central measure $P
_{0}$ in $\mathcal{H}_{\infty}(\Omega)$ admits a unique integral representation
\[
P_{0}=\int_{\partial\mathcal{H}_{\infty}(\Omega)}P d\mu(P),
\]
where $\mu$ is a probability measure on the set $\partial\mathcal{H}_{\infty
}(\Omega)$. The set $\partial \mathcal{H}_{\infty}(\Omega)$ is called the minimal boundary of $\Gamma(\Omega)$.

\subsection{The graph embedding and Martin theory}

Let $P\in\mathcal{H}_{\infty}(\Omega)$. Then, by Definition \ref{defHarmRandompath}
there exists a function $p:\Lambda\times\mathbb{N}\longrightarrow
\mathbb{R}^{+}$ such that
\[
P(\Gamma_{\Omega}(\gamma))=p(\gamma(l(\gamma
)),l(\gamma)),
\]
for all $\gamma\in\Gamma_{f}(\Omega)$. Let $\lambda\in\Lambda$, and suppose that
$\gamma$ is a finite path of $\Gamma_{f}(\Omega)$ starting at $x$ and ending
at $\lambda$ with length $n$. A path $\tau$ of length $n+1$ ending at $\mu
\in\Lambda$ satisfies $\tau_{\downarrow n}=\gamma$ if and only if
$\tau_{\downarrow n}=\gamma$ and $\tau_{[n,n+1]}$ is an infinitesimal path joining
$\lambda$ to $\mu$. Therefore, $\Gamma_{\Omega}(\gamma)$ can be decomposed as
\[
\Gamma_{\Omega}(\gamma)=\coprod_{\mu\in\Lambda}\coprod_{\tau\in\Gamma
{\lambda,\mu}(\Omega)}\Gamma(\gamma.\tau).
\]
Thus,
\[
P\left( \Gamma_{\Omega}(\gamma)\right)  =\sum_{\mu
\in\lambda}\sum_{\tau\in\Gamma_{\Omega}(\lambda,\mu)}P\left(\Gamma(\gamma.\tau)\right)  ,
\]
which translates into the relation
\begin{equation}
p(\lambda,n)=\sum_{\mu\in\Lambda}\#\Gamma_
{\Omega}(\lambda,\mu)p(\mu,n+1),
\label{recursiveRelation}%
\end{equation}
where in the latter equality and in the sequel of the paper the cardinality of a set $X$ is denoted by $\#X$.
The set $\mathcal{H}_{\infty}(\Omega)$ is in bijection with the set of nonnegative
solutions of \eqref{recursiveRelation} with value $1$ on $(0,0)$. This
equivalence leads to an alternative description of central random paths.

\begin{definition}
The growth graph of $\Gamma(\Omega)$ is the rooted graded graph $\mathcal{G}(\Omega)$
defined recursively as follows:

\begin{itemize}
\item The root is denoted by $(0,0)$.

\item For each element $\lambda$ of $\Lambda$ such that there exists an
infinitesimal path ending at
$\lambda$, we define a vertex $(\lambda,1)$ of rank $1$ and an edge between
$(0,0)$ and $(\lambda,1)$ with weight $e(x,\lambda)=\#\Gamma
_{0,\lambda}(\Omega)$.

\item Let $n\geq1$, and suppose that the graded graph is defined up to rank
$n$: the set $\mathcal{G}_{n}(\Omega)$ of vertices of rank $n$ can be written as
$\lbrace(\lambda,n)\rbrace_{\lambda\in\Lambda_{n}}$, where $\Lambda_{n}$ is a
subset of $\Lambda$. For each element $\mu$ of $\Lambda$ such
that there exists an infinitesimal path $\gamma$ with $\gamma(0)\in\Lambda
_{n}$ and $\gamma(1)=\mu$, we define a vertex $(\mu,n+1)$ of rank $n+1$. For
each $\lambda\in\Lambda_{n}$ there is an edge from $(\lambda,n)$ to
$(\mu,n+1)$ with weight $e(\lambda,\mu)=\#\Gamma_{\Omega}(\lambda,\mu)$.
\end{itemize}
\end{definition}

We write $\lambda\nearrow\mu$ when $\#\Gamma_{\Omega}(\lambda,\mu)\not =0$. It is
readily seen that the number of paths between the root and $(\lambda,n)$ is
canonically equal to $\#\Gamma_{\Omega}(\lambda,n)$, and the set $\mathcal{H}%
_{\infty}(\Omega)$ is isomorphic to the set of nonnegative functions $p:\coprod
_{n\geq0}\Lambda_{n}\longrightarrow\mathbb{R}^{+}$ with $p(0,0)=1$ and
$p(\lambda,n)=\sum_{\lambda\nearrow\mu}e(\gamma,\mu)p(\mu,n+1)$.

We conclude
this subsection by establishing some connections between central measures on
random paths and Markov chains on lattices. From the growth graph of $\Gamma(\Omega)$, it is clear
that any central measure $P\in\mathcal{H}_{\infty}(\Omega)$ yields a Markov
chain $Z=(Z(0),Z(1),\dots)$ on the lattice
$\Lambda\cap\Omega$ with initial state $0$ and with a family of Markov kernels $(Q_{n})_{n\geq1}$: the
kernel $Q_{n}$ can be explicitly given from the
associated function $p:\coprod\Lambda_{n}\longrightarrow\mathbb{R}^{+}$ as
\[
Q_{n}(\mu,\nu)=\mathbf{1}_{\mu\nearrow\nu,p(\mu,n-1)\not=0}\frac{e(\mu,\nu)p(\nu
,n)}{p(\mu,n-1)}.
\]
By the equality $p(\mu,n-1)=\sum_{\mu\nearrow\nu}e(\mu,\nu)p(\nu,n)$, $Q_{n}$
is a well-defined Markov kernel, and it is readily seen that this family of
Markov kernels generates the random walk $Z$. Note that this random walk is generally not homogeneous in times, since the kernel $Q_{n}$ depends on $n$ through $p$.

\subsection{Doob conditioning and central measure}\label{doobconditioning}
Let $(\omega_{t})_{t\geq 0}$ be the random path in $\Gamma(\mathbb{R}^{d})$ defined by the Markovian evolution 
$$\mathbb{P}(\omega_{\vert [i,i+1]}-\omega(i)=\gamma)=\frac{1}{Z}$$
where $Z$ is the cardinality of $X$. Under this perspective, the central measures defined on $\Gamma(\Omega)$ in the previous section are exactly the possible ways to condition the random path $\omega$ to stay in $\Omega$, while keeping the uniform law on each set $\Gamma_{\Omega}(x,n)$.

The standard procedure to achieve this is the Doob-conditioning with $c$-harmonic functions. Assume that $h$ is a $c$-harmonic function (with $c>0$) for the random path $\omega$ killed when exiting $\Omega$. Namely, we have
$$h(x)=\frac{1}{cZ}\sum_{\substack{\gamma\in X\\x+\gamma\subset\Omega}}h(x+\gamma(1)).$$
Then, the Doob conditioning $\omega^{h}$ of $\omega$ in $\Omega$ is given by the Markov kernel
$$\mathbb{P}(\omega^{h}_{\vert [i,i+1]}-\omega^{h}(i)=\gamma\vert \omega^{h}(i)=x )=\mathbf{1}_{x+\gamma\subset \Omega}\frac{1}{cZ}\frac{h(x+\gamma(1))}{h(x)},$$
and we have $\mathbb{P}(\omega^{h}\in\Gamma_{\Omega}(\gamma))=\frac{1}{(cZ)^{n}}h(x)$ for all $\gamma\in\Gamma_{\Omega}(x,n)$. The resulting random path $\omega^{h}$ is thus central, and the associated function $p_{\omega^{h}}$ is exactly $p_{\omega^{h}}(x,n)=\frac{1}{(cZ)^{n}}h(x)$.

Reciprocally, suppose that $\omega$ is a central random path with an associated function $p_{\omega}$ which satisfies $p_{\omega}(x,n)=\frac{1}{K^{n}}p(x)$ for some function $p:\Lambda\cap\Omega\longrightarrow \mathbb{R}^{+}$ and $K>0$. Then, \eqref{recursiveRelation} yields 
$$\frac{1}{K^{n}}p(x)=\sum_{y\in\Lambda}\#\Gamma_{\Omega}(x,y)\frac{1}{K^{n+1}}p(y),$$
which is equivalent to the relation
$$p(x)=\sum_{\substack{\gamma\in X\\x+\gamma\subset \Omega}}\frac{1}{K}p(x+\gamma(1)).$$
Hence, $p$ is a $\frac{K}{Z}$-harmonic function. 

Thus, the set of $c$-harmonic functions is homeomorphic to the set of central random paths $\omega\in\mathcal{H}_{\infty}(\Omega)$ whose associated functions $p_{\omega}$ have the form $p_{\omega}(x,n)=\frac{p(x)}{(cZ)^{n}}$ with $p:\Lambda\cap\Omega\longrightarrow \mathbb{R}^{+}$. We denote by $\mathcal{H}_{c}(\Omega)$ the set of central measures coming from $c$-harmonic functions, and by $\partial \mathcal{H}_{c}(\Omega)$ the set of extreme points of $\mathcal{H}_{c}(\Omega)$. Up to our knowledge there is no general proof that $\partial \mathcal{H}_{c}(\Omega)=\partial\mathcal{H}_{\infty}(\Omega)\cap\mathcal{H}_{c}(\Omega)$; in our case of study, this equality is proven by explicitly describing both sets.

We remark that the random walk $Z_{\omega}$ associated with a central random path $\omega\in\mathcal{H}_{c}(\Omega)$ is homogeneous in time. A quick computation shows that $Z_{\omega}$ is homogeneous in time if and only if $\omega\in\mathcal{H}_{c}(\Omega)$ for some $c>0$.

\subsection{Central measures on multiplicative graphs}

A rooted graded graph $\mathcal{G}=\{\ast\}\sqcup\coprod_{n\geq1}%
\mathcal{G}_{n}$ with weights $(e(\lambda,\mu))_{\substack{\mu,\lambda
\in\mathcal{G}\\\lambda\nearrow\mu}}$ is called multiplicative if there is a commutative
algebra $A$ and an injective map $i:\mathcal{G}\longrightarrow A$ such that
$i(\lambda)i(\ast)=\sum_{\lambda\nearrow\mu}e(\lambda,\mu)i(\mu)$. We suppose
that the graph is connected, which means that for all $\mu\in\mathcal{G}$, the
number of paths between the root and $\mu$ is positive. The weight $w(\gamma)$
of a path $\gamma$ between the root and a vertex $\mu$ is the product of all
the weights of the edges of $\gamma$. Let $K$ be the positive cone spanned by
$i(\mathcal{G})$, and let $A_{\mathcal{G}}$ be the unital subalgebra of $A$
generated by $K$. The following result is an application of the Ring theorem
of Kerov and Vershik (see for example \cite[Section 8.4]{GO}) which
characterizes the extreme points of the set $\mathcal{H}(\mathcal{G})$ of
solutions to the following problem:
\begin{equation}
\left\{
\begin{matrix}
& p:\mathcal{G}\longrightarrow\mathbb{R}^{+}\\
& p(\ast)=1\\
& p(\lambda)=\sum_{\lambda\nearrow\mu}e(\gamma,\mu)p(\mu).
\end{matrix}
\right.  \label{recursiveRelationBis}%
\end{equation}
Denote by $\Mult^{+}(A_{\mathcal{G}})\subset A_{\mathcal{G}}^{\ast}$ the set of
multiplicative functions on $A_{\mathcal{G}}$ which are nonnegative on $K$
and equal to $1$ on $i(\ast)$. Note that $i:\mathcal{G}\longrightarrow
A_{\mathcal{G}}$ induces a map $i^{\ast}:A_{\mathcal{G}}^{\ast}\longrightarrow
F(\mathcal{G},\mathbb{R})$.

\begin{proposition}
\label{multiGraphExtreme} Suppose that $K.K\subset K$. Then, the map $i^{*}$
yields an homeomorphism between $\Mult^{+}(A_{\mathcal{G}})$ and the set of
extreme points of $\mathcal{H}(\mathcal{G})$.
\end{proposition}

The proof of this proposition is based on the following Theorem of Kerov and Vershik:

\begin{theorem}\cite[Section 8.4]{GO}
\label{VershKer} Let $B$ be a unital commutative algebra over $\mathbb{R}$ and
$K\subset B$ a convex cone satisfying the following conditions:

\begin{itemize}
\item $K-K=B$ ($K$ generates $B$).

\item $K.K\subset K$ ($K$ is stable by multiplication).

\item $K$ is spanned by a countable set of elements.

\item For all $a\in B$, there exists $\epsilon>0$ such that $1-\epsilon a\in
K$.
\end{itemize}

If $L$ denotes the convex set of linear forms on $B$ which are nonnegative on
$K$ and map $1_{B}$ to $1$, then $\phi$ is an extreme point of $L$ if and only
if $\phi$ is multiplicative (meaning that $\phi(ab)=\phi(a)\phi(b)$ for all
$a,b\in B$).
\end{theorem}

We give now the proof of Proposition \ref{multiGraphExtreme}.

\begin{proof}
Let $B=A_{\mathcal{G}}/{\small        {\langle i(*)=1\rangle}}$ and let $\pi:A_{\mathcal{G}%
}\longrightarrow B$ be the canonical projection; denote by $\tilde{K}$ the
projection of the cone $\mathbb{R}^{+}\Id +K$ in $B$. Since $K.K\subset K$ and
$\lbrace1,K\rbrace$ spans $A_{\mathcal{G}}$, $\tilde{K}.\tilde{K}\subset
\tilde{K}$ and $\tilde{K}$ spans $B$. Since $\mathcal{G}$ has a countable set
of vertices, $\tilde{K}$ is spanned by a countable set of elements. Note that
there is a bijection between the elements of $\mathcal{H}(\mathcal{G})$ and
the linear forms on $B$ which are nonnegative on $\tilde{K}$ and equal to $1$
on $1$: indeed $h\in\mathcal{H}(\mathcal{G})$ if and only if $h(\mu)=\sum
_{\mu\nearrow\nu}e(\mu,\nu)h(\nu)$. Thus, for $f\in A_{\mathcal{G}}^{*}$,
$i^{*}(f)\in\mathcal{H}(\mathcal{G})$ if and only if $f(i(*)i(\mu))=f(i(\mu
))$; equivalently, this means that $f$ factors through $B$. Non-negativeness
on $\mathcal{G}$ for $i^{*}(f)$ is then equivalent to nonnegativeness on
$\tilde{K}$ for $f$, and $[i^{*}(f)](*)=1$ if and only if $f(\pi\circ
i(*))=f(1)=1$.

Let $a\in B$, and let us show that there exists
$\epsilon$ such that $1-\epsilon a\in\tilde{K}$. Since $\tilde{K}-\tilde{K}%
=B$, and $1-b\in\tilde{K}$ for all $b\in-\tilde{K}$, we can suppose without
loss of generality that $a\in\tilde{K}$. It is thus enough to prove that for
$\mu\in\mathcal{G}$, there exists $\epsilon$ such that $1-\epsilon\pi\circ
i(\mu)\in K$. Suppose that $\mu$ has rank $n$. Since the graph is connected,
there exists a path $\gamma_{0}$ of weight $w(\gamma_{0})$ between $*$ and
$\mu$. By iteration of the relation coming from the multiplicative structure
of $\mathcal{G}$, $i(*)^{n}=\sum_{\substack{\nu\in\mathcal{G}\\rk(\mu
)=n}}(\sum_{\gamma:*\rightarrow\mu}w(\gamma))i(\nu)$. Thus $i(*)^{n}%
-w(\gamma_{0})i(\mu)$ belongs to $K$. Since $\pi(i(*)^{n})=1$, $1-w(\gamma
_{0})\pi\circ i (\mu)$ belongs to $\tilde{K}$. Therefore, we can apply
Theorem \ref{VershKer} to $(B,\tilde{K})$, which yields that the extreme
linear maps among the set of linear maps on $B$ which are nonnegative on
$\tilde{K}$ and equal to $1$ on $1$ are the multiplicative ones. Since there
is a bijection between multiplicative maps on $B$ which are nonnegative on
$\tilde{K}$ and multiplicative maps on $A_{\mathcal{G}}$ which are nonnegative
on $K$ and equal to $1$ on $i(*)$, the proof is complete.
\end{proof}

\section{Littelmann paths in Weyl chambers}

We describe a class of random paths coming from the representation theory of
semi-simple Lie groups.

\subsection{Background}

We consider a simple Lie group $G$ over $\mathbb{C}$ and its Lie algebra
$\mathfrak{g}$.\ Let $R\subset V$ be the set of roots of $\mathfrak{g}$
regarded as a finite subset of the euclidean vector space $V$. We fix $R_{+}$
a subset of positive roots and $S=\{\alpha_{1},\alpha_{2},\ldots,\alpha
_{d}\}\subset R_{+}$ a basis of $R$. The Weyl group of $\mathfrak{g}$ is
denoted by $W$.

Write $P$ for the weight lattice of $\mathfrak{g}$ and $\omega_{1}%
,\ldots,\omega_{d}$ for its fundamental weights so that we have
\[
P={\bigoplus\limits_{i=1}^{d}}\mathbb{Z\omega}_{i}.
\]
We denote by $\Delta$ the fundamental Weyl chamber of $\mathfrak{g}$ with
respect to $S$, which corresponds to the positive orthant on the weight space
$\bigoplus_{i=1}^{d}\mathbb{R}\omega_{i}$. The cone of dominant weights is
then
\[
P^{+}=P\cap\Delta={\bigoplus\limits_{i=1}^{d}}\mathbb{N\omega}_{i}.
\]
Write $Q^{+}$ the subset of $P$ spanned by linear combinations of the
simple roots with nonnegative coefficients. We denote by
$\mathbb{R}[P]$ the ring group of $P$ over $\mathbb{R}$ with basis
$\{e^{\beta}\mid\beta\in P\}$, and by $\mathbb{R}[Q^{+}]$ the subalgebra of
$\mathbb{R}[P]$ generated by $Q^{+}$. Then
\[
\mathbb{R}^{W}[P]=\{u\mid w(u)=u,w\in W\}
\]
is the character ring of $\mathfrak{g}$.\ To each $\lambda\in P_{+}$
corresponds a simple finite-dimensional representation of $\mathfrak{g}$ we
denote by $V(\lambda)$.\ The Weyl character of $V(\lambda)$ is
\[
s_{\lambda}=\sum_{\gamma\in P}K_{\lambda,\gamma}e^{\gamma}%
\]
where $K_{\lambda,\gamma}$ is the dimension of the weight space $\gamma$ in
$V(\lambda)$. For $\vec{t}\in(\mathbb{R}^{+})^{d}$ and $\gamma\in \mathbb{R}^{d}$, set $\vec{t}^{\gamma}=\prod_{1\leq i\leq d}\exp(\gamma_{i}\log(t_{i}))$. It is then possible to define the evaluation of $s_{\lambda}$ on $\vec{t}\in(\mathbb{R}^{+})^{d}$ with the formula $s_{\lambda}(\vec{t})=\sum_{\gamma\in P} K_{\lambda,\gamma}\vec{t}^{\gamma}$. 
For $\mu\geq\lambda$ (that is $\mu-\lambda$ is a sum of simple roots),
denote by $S_{\lambda,\mu}$ the function
\[
S_{\lambda,\mu}=e^{-\mu}s_{\lambda}=\sum_{\gamma\in P}K_{\lambda,\gamma
}e^{\gamma-\mu}%
\]
where for any $\gamma$ such that $K_{\lambda,\gamma}>0$, $\gamma-\mu$ is a
linear combination of the simple roots with nonpositive coefficients; for
$\mu=\lambda$, we simply write $S_{\lambda}$, instead of $S_{\lambda,\lambda}%
$. By setting $T_{i}=e^{-\alpha_{i}}$ we thus obtain that $S_{\lambda,\mu
}=S_{\lambda,\mu}(T_{1},\ldots,T_{d})$ is polynomial in the variables
$T_{1},\ldots,T_{d}$ with nonnegative integer coefficients. Recall also the
Weyl dimension formula%
\[
\dim(V(\lambda))=\prod_{\alpha\in R_{+}}\frac{(\lambda+\rho,\alpha)}%
{(\rho,\alpha)}.%
\]
In particular, $\dim(V(\lambda))$ is polynomial in the coordinates of
$\lambda$ on the basis of fundamental weights.

\subsection{Random Littelmann paths}

Now, fix a dominant weight $\delta\in P^{+}$ and denote by $\Pi_{\delta}$ the
set of weights of the irreducible representation $V(\delta)$. Let $P_{\delta}$
be the sublattice of $P$ generated by $\Pi_{\delta}$.\ This defines
subalgebras
\[
\mathbb{R}[P_{\delta}]=\{e^{\beta}\mid\beta\in P_{\delta}\}\subset
\mathbb{R}[P]\text{ and }\mathbb{R}^{W}[P_{\delta}]=\{u\in\mathbb{R}%
[P_{\delta}]\mid w(u)=u\}\subset\mathbb{R}^{W}[P].
\]
Finally write $T_{\delta}^{+}$ the subset of $P^{+}$ of weights $\lambda$
such that $V(\lambda)$ appears as an irreducible component in a tensor power
$V(\delta)^{\otimes n},n\geq0$. Given $\lambda$ and $\mu$ in $T_{\delta
}^{+},$ we clearly have $\lambda+\mu$ in $T_{\delta}^{+}$. Moreover the
$\mathbb{Z}$-lattice $T_{\delta}$ generated by $T_{\delta}^{+}$ is a
sublattice of $P_{\delta}$.\ We thus have the following inclusions of
$\mathbb{Z}$-lattices%
\[
T_{\delta}\subset P_{\delta}\subset P.
\]
Since $B=(\omega_{1},\ldots,\omega_{d})$ is a $\mathbb{Z}$-basis of $P_{+}$
there exists $(q_{1},\ldots,q_{d})\in\mathbb{Z}_{>0}$ such that $q_{i+1}\mid
q_{i}$ for any $i=1,\ldots,d-1$ and
\[
P_{\delta}={\displaystyle\bigoplus\limits_{i=1}^{d}} \mathbb{Z}_{\geq0}%
q_{i}\mathbb{\omega}_{i}.
\]
Now let $\mathcal{A}_{\delta}$ be the subalgebra of $\mathbb{R}^{W}[P]$
generated by the Weyl character $s_{\lambda}$ with $\lambda\in T_{\delta}^{+}%
$.\ We have the inclusions%
\[
\mathcal{A}_{\delta}\subset\mathbb{R}^{W}[P_{\delta}]\subset\mathbb{R}%
[P_{\delta}]\subset\mathbb{R}[P]\text{.}%
\]
We denote by $K(\delta)$ the convex hull of the set $\Pi_{\delta}$:
$K(\delta)$ is a polytope whose extreme points are the elements $w(\delta)$
for $w\in W$. The intersection of $K(\delta)$ with the Weyl chamber $\Delta$
is denoted by $K(\delta)^{+}$. By Littelmann's paths theory, there is a set
$B(\delta)=\{\gamma_{i}\}_{1\leq i\leq\dim V(\delta)}$ of infinitesimal paths
on $P_{\delta}$, with the following properties:

\begin{itemize}
\item $\gamma_{i}(1)\in\Pi_{\delta}$ for all $1\leq i\leq\dim V(\delta)$.

\item The multiplicity of the weight $\mu$ in $V(\delta)^{\otimes n}$ is equal to
$\#\Gamma_{\mathbb{R}^{d}}(\mu,n)$.

\item The multiplicity of the irreducible representation $V(\nu)$ in $V(\mu)\otimes V(\delta)$ is equal to $\#\Gamma_{\Delta}(\mu,\nu)$ and the multiplicity of the irreducible representation $V(\nu)$ in
$V(\delta)^{\otimes n}$ is equal to $\#\Gamma_{\Delta
}(\nu,n)$ for all $\mu,\nu\in P^{+}$ and $n\geq0$.
\end{itemize}

The set of infinite paths we are interested in is the set of infinite paths starting at $0$ with set of infinitesimal paths $B(\delta)$.

\subsection{Statements of the result}

We recall that we consider the space of
probability measures on each $\Gamma(\Omega)$ with the initial topology with
respect to the evaluation maps on the cylinders $\Gamma_{\Omega}(\tau)$,
$\tau\in\Gamma_{f}(\Omega)$. We give an algebraic proof of the 
identification of the minimal boundaries for random paths in $\Gamma
(\mathbb{R}^{d})$ and $\Gamma_{\infty}(\Delta)$ with the topological spaces
$K(\delta)$ and $K(\delta)^{+}$, respectively. In both cases, the
homeomorphism can be made explicit by the introduction of a natural
parametrization $t:K(\delta)\longrightarrow\lbrack0,1]^{d}\times W$ of
$K(\delta)$ such that $t(K(\delta)^{+})\subset\lbrack0,1]^{d}\times \Id_{W}$
(this parametrization is explained in Section $5$). For $m\in K(\delta)$, we
denote by $(\vec{t}_{m},w_{m})$ the image of $m$ through this parametrization.
The main result of the paper is summarized in the following theorem: 

\begin{theorem}
\label{mainresult} A homeomorphism between the set of extremal measures $\partial\mathcal{H}_{\infty
}(\mathbb{R}^{d})$ and $K(\delta)$ is given by the map
\[
\mathbb{P}:\left\{
\begin{matrix}
K(\delta) & \longrightarrow & \partial\mathcal{H}_{\infty}(\mathbb{R}^{d})\\
m & \mapsto & \mathbb{P}_{m}%
\end{matrix}
\right.
\]
such that $\mathbb{P}_{m}(\Gamma_{\mathbb{R}^{d}}(\gamma
))=\frac{\vec{t}_{m}^{N\delta-w_{m}(\lambda)}}{S_{\delta}(\vec{t}_{m})}$ for
all $\gamma\in\Gamma_{\mathbb{R}^{d}}(\lambda,N)$. A homeomorphism between the set
of extremal measures $\partial\mathcal{H}_{\infty}(\Delta)$ and $K(\delta)^{+}$ is given by the map
\[
\mathbb{P}^{+}\left\{
\begin{matrix}
K(\delta)^{+} & \longrightarrow & \partial\mathcal{H}_{\infty}(\Delta)\\
m & \mapsto & \mathbb{P}_{m}^{+}%
\end{matrix}
\right.
\]
such that $\mathbb{P}_{m}^{+}(\Gamma_{\Delta}(\gamma))=\frac
{S_{\lambda,N\delta}(\vec{t}_{m})}{S_{\delta}(\vec{t}_{m})^{N}}$ for all
$\gamma\in\Gamma_{\Delta}(\lambda,N)$.
\end{theorem}

It is easy to see that the measures $\mathbb{P}_{m}$ and $\mathbb{P}_{m}^{+}$ are indeed central.
Note moreover that for $m\in K(\delta)^{+}$, Littelmann's theory yields that for $\gamma
\in\Gamma_{\Delta}(y,N)$,
\begin{align*}
\sum_{\tilde{\gamma}\in\Gamma_{N+1}(\Delta),\tilde{\gamma}_{\downarrow N}=\gamma
}\mathbb{P}_{m}^{+}(\Gamma_{\Delta}(\tilde{\gamma}))=  &  \frac
{\sum_{\mu\in B(\delta),\gamma.\mu\in\Gamma_{\Delta}(y,N+1)}S_{\gamma(N)+\mu
(1),N\delta+x}(\vec{t}_{m})}{S_{\delta}(\vec{t}_{m})^{N+1}}\\
=  &  \frac{S_{\gamma(N),N\delta}(\vec{t}_{m})S_{\delta}(\vec{t}_{m}%
)}{S_{\delta}(t_{m})^{N+1}}\\
=  &  \frac{S_{\gamma(N),N\delta}(t_{1},\dots,t_{d})}{S_{\delta}(\vec{t}%
_{m})^{N}}\\
=  &  \mathbb{P}_{m}^{+}(\Gamma_{\Delta}(\gamma)),
\end{align*}
so that $\mathbb{P}_{m}^{+}$ is a well defined probability measure on
$\Gamma_{\infty}(\Delta)$. The main point of the result is to prove that $\mathbb{P}$
and $\mathbb{P}^{+}$ are bijective.

\begin{remark}
In type $A_{d}$, when $\delta=\omega_{1}$ is the first fundamental weight,
$V(\delta)$ can be regarded as the defining representation of $\mathfrak{sl}%
_{d+1}$ or more conveniently, of $\mathfrak{gl}_{d+1}$. The set $\partial
\mathcal{H}_{\infty}(\Delta)$ is then homeomorphic to
\[
K(\delta)^{+}=\{(p_{1},\ldots,p_{n+1})\in\mathbb{R}^{d+1}\mid p_{1}\geq
\cdots\geq p_{n+1}\geq0\text{ and }p_{1}+\cdots+p_{n+1}=1\}
\]
and we recover the finite-dimensional version of the Thoma simplex.
\end{remark}
As a corollary of Theorem \ref{mainresult}, we get the complete characterization of $c$-harmonic measures killed when exiting $\Delta$. Define the function $\hat{s}_{\delta}:\partial \mathcal{H}_{\infty}(\Delta)\longrightarrow \mathbb{R}^{+}\cup\lbrace \infty\rbrace$ by $\hat{s}_{\delta}(\mathbb{P}_{m})=s_{\delta}(\vec{t}_{m})$.
\begin{corollary}\label{corollary}
For $c>0$, the set $\partial\mathcal{H}_{c}(\Delta)$ is homeomorphic to $\hat{s}_{\delta}^{-1}(\lbrace c\dim V(\delta)\rbrace)$. In particular,
\begin{itemize}
\item $\mathcal{H}_{1}(\Delta)$ is a singleton corresponding to $\mathbb{P}_{\vec{0}}$,
\item and for $c<1$, $\mathcal{H}_{c}(\Delta)=\emptyset$.
\end{itemize}
\end{corollary}
We prove Corollary \ref{corollary} in Section \ref{proofCorollary}. We discuss here a possible generalization of the latter result. Let $X$ be an arbitrary set of infinitesimal paths, and let $wg:X\longrightarrow \mathbb{R}^{+}$ be a weight function. We denote by $\hat{Z}(\vec{t}):=\sum_{\gamma\in X}wg(\gamma)\vec{t}^{\gamma(1)}$ the partition function for this weighting, and we simply write $Z$ for $\hat{Z}(1)$. Finally, let us fix a cone $\mathcal{C}$ centered at $0$ and denote by $K_{\mathcal{C}}$ the set of elements $\vec{t}\in\mathbb{R}^{d}$ such that $\sum_{\gamma\in X}wg(\gamma)\vec{t}^{\gamma(1)}\gamma(1)\in \mathcal{C}$.\\
The definition of $c$-harmonic functions for weighted paths is similar to the one for unweighted paths; namely, a function $f$ is $c$-harmonic if and only if 
$$f(x)=\sum_{\substack{\gamma\in X\\x+\gamma\subset \mathcal{C}}}\frac{1}{cZ}wg(\gamma)f(x+\gamma(1)).$$
We use the same notation as in the unweighted case, since the theory is the same in this broader situation. Then, we conjecture that the following general result holds:
\begin{conjecture}
For $c>0$, the set $\partial\mathcal{H}_{c}(\Delta)$ is homeomorphic to $\hat{Z}^{-1}(\lbrace cZ\rbrace)\cap K_{\mathcal{C}}$. In particular,
\begin{itemize}
\item for $u=\min_{K_{\mathcal{C}}} \hat{Z}$, $\mathcal{H}_{u/Z}(\Delta)$ is a singleton.
\item and for $c<u/Z$, $\mathcal{H}_{c}(\Delta)=\emptyset$.
\end{itemize}
\end{conjecture}
This conjecture is a generalization of the conjecture of Raschel \cite[Conjecture 1]{R} for two dimensional random walks with bounded increments, which asserts that such a random walk admits a unique harmonic function killed on the boundary of a quarter plane. This special situation can be seen in the above conjecture, in which case the minimum of $\hat{Z}$ is exactly $Z$.
\subsection{The approach of Handelman and Price}\label{explanationProof}
The existence of the homeomorphism of Theorem \ref{mainresult} can also be deduced from the main results of \cite{Han1, Han2}, themselves based on fundamental results of \cite{Pr1,Pr2}. We review here their approach, and the reader could read the aforementioned articles and references therein for a detailed proof. 

Let $n$ denote the dimension of $V(\delta)$, and consider the adjoint representation $\pi:G\longrightarrow GL(M_{n}(\mathbb{C}))$ which is defined by $\pi(g)(M)=u_{\delta}(g)Mu_{\delta}(g)^{-1}$, where $u_{\delta}$ is the irreducible representation associated with $\delta$. Form the infinite tensor product $A:=\bigotimes M_{n}(\mathbb{C})$ as an inductive limit of the sequence of finite-dimensional $C^{*}$-algebras $(M_{n}(\mathbb{C})^{\otimes k})_{k\geq 1}$, where $M_{n}(\mathbb{C})^{\otimes k}$ embeds in $M_{n}(\mathbb{C})^{\otimes k+1}$ with the map $X\mapsto X\otimes \Id_{n}$. We can canonically associate a structure of $C^{*}$-algebra to this inductive limit of $C^{*}$-algebras. Then, $G$ acts continuously on each $M_{n}(\mathbb{C})^{\otimes k}$ and on $A$ with the map $\tilde{\pi}(g):=\bigotimes \pi(g)$ (which means that $g$ acts as $\pi(g)$ on each component of the tensor product), and we can therefore consider the $C^{*}$-algebra $A^{\delta}$ (resp. $A^{\delta}_{k})$ of elements of $A$ (resp. $M_{n}(\mathbb{C})^{\otimes k}$) fixed by $\pi$. The algebra $A^{\delta}$ is the inductive limit of the finite-dimensional $C^{*}$-algebras $(A^{\delta}_{k})_{k\geq 1}$, and the Bratteli diagram of this inductive limit is exactly the growth graph of $\Gamma_{\infty}(\Delta)$. Therefore, the set of central measures on $\Gamma_{\infty}(\Delta)$ is in bijection with the set of traces on $A^{\delta}$.\\
Doing the same construction for the restriction of the representation $\delta$ to the maximal torus $T\subset G$, we get another sequence of finite dimensional $C^{*}$-algebras $(A_{k}^{T})_{k\geq 1}$, whose inductive limit is denoted by $A^{T}$. Similarly, the Bratteli diagram of $A^{T}$ is exactly the growth graph of $\Gamma_{\infty}(\mathbb{R}^{d})$, and the set of central measures on $\Gamma_{\infty}(\mathbb{R}^{d})$ is in bijection with the set of traces on $A^{T}$.\\
Note that we have the natural inclusion of $C^{*}$-algebras $A^{\delta}\subset A^{T}$. The main result of \cite{Han1} is that any extreme trace on $A^{\delta}$ extends to an extreme trace on $A^{T}$. To prove this, the author uses the bijection between the set of traces on an approximately finite $C^{*}$-algebra $A$ and the set of states on its associated dimension group $K_{0}(A)$. Let us quickly explain the nature of $K_{0}(A)$: a dimension group is a group with a notion of positive cone. By considering equivalence classes of projections on the $*$-algebra $\bigoplus_{k\geq 1} M_{k}(A)$, one can canonically associate a dimension group $K_{0}(A)$ to each $C^{*}$-algebra $A$; this dimension group is always a ring in our case. An important fact is that an inclusion of $C^{*}$-algebras induces an inclusion of the associated dimension groups, and therefore the problem reduces to extend any state on $K_{0}(A^{\delta})$ to a state on $K_{0}(A^{T})$. Handelman managed to prove this in \cite{Han1}, and the main ingredient of the proof is the non-trivial property that $K_{0}(A^{T})$ is a finitely generated $K_{0}(A^{\delta})$-module.\\
Once proven that any trace on $A^{\delta}$ extends to a trace on $A^{T}$, the problem amounts to describe the set of traces on $A^{T}$. In \cite{Han2}, the author achieves this by proving that the set of faithful traces on $A^{T}$ is in bijection with the interior of $K(\delta)$.
Then, the identification of the set of faithful traces on $A^{\delta}$ with the interior of $K(\delta)^{+}$ is done thanks to a result of \cite{Pr2}, which asserts that the Weyl group $W$ acts transitively on the set of traces extending a particular faithful trace on $A^{T}$. Finally, the case of non-faithful traces is done by considering parabolic subgroups of $G$. 
\subsection{The extended algebra of characters}

Our proof of Theorem \ref{mainresult} will mainly use algebraic properties of the
representations of the Lie algebra $\mathfrak{g}$. We define the extended
algebra of characters $\hat{\mathcal{A}}_{\delta}$ as follows:

\begin{itemize}
\item $\hat{\mathcal{A}}_{\delta}$ is isomorphic to $\mathcal{A}_{\delta
}\times\mathbb{R}[T]$ as a vector space; for $x\in A_{\delta}$, we simply
denote by $(x,n)$ the element $(x,T^{n})$. A basis of $\hat{\mathcal{A}%
}_{\delta}$ is given by the set $\mathcal{B}=\{(s_{\lambda},n)\}_{n\geq
0,\lambda\in T_{\delta}^{+}}$.

\item The multiplicative structure of $\hat{\mathcal{A}}_{\delta}$ is defined
on $\mathcal{B}$ with the product
\[
(s_{\lambda},n)\times(s_{\mu},m)=(s_{\lambda}s_{\mu},n+m).
\]

\end{itemize}

We denote by $\hat{T}^{+}_{\delta}$ the subalgebra of $\hat{\mathcal{A}%
}_{\delta}$ spanned by the set $\lbrace(s_{\lambda},n)\vert V(\lambda)\in
V(\delta)^{\otimes n}\rbrace$. Here $V(\lambda)\in V(\delta)^{\otimes n}$ means that $V(\lambda)$ is an irreducible component of $V(\delta)^{\otimes n}$. Likewise, we define the extended algebra of
weights $\hat{P}_{\delta}$ as follows

\begin{itemize}
\item $\hat{P}_{\delta}$ is isomorphic to $\mathbb{R}[P_{\delta}]\times
\mathbb{R}[T]$ as a vector space. A basis of $\hat{P}_{\delta}$ is
given by the set $\mathcal{P}=\lbrace(e^{\gamma},n)\vert n\geq0,\gamma
\in P_{\delta} \rbrace$.

\item The multiplicative structure of $\hat{P}_{\delta}$ is defined on
$\mathcal{P}$ with the product
\[
(e^{\gamma},n)\times(e^{\gamma^{\prime}},m)=(e^{\gamma+\gamma^{\prime}},n+m).
\]

\end{itemize}

We denote by $\hat{T}_{\delta}$ the subalgebra of $\hat{P}_{\delta}$ spanned
by the elements $\lbrace(e^{\gamma},n)\vert n\geq1, K_{\delta^{\otimes
n},\gamma}>0\rbrace$. It is readily seen that $\hat{T}_{\delta}$ is the subalgebra of
$\hat{P}_{\delta}$ generated by $\lbrace(e^{\gamma},1)\vert\gamma\in
\Pi_{\delta}\rbrace$. Note that the inclusion $\mathcal{A}_{\delta
}\subset\mathbb{R}[P_{\delta}]$ translates naturally into the inclusion
$\hat{\mathcal{A}}_{\delta}\subset\hat{P}_{\delta}$ and $\hat{T}^{+}_{\delta
}\subset\hat{T}_{\delta}$. We can write the multiset of weights of $\delta$ in
$\hat{T}_{\delta}$ as $\Pi_{\delta}=\{(e^{\gamma_{1}},1),\ldots,(e^{\gamma
_{N}},1)\}$ where each weight appears a number of times equal to its
multiplicity. For any $k=0,\ldots,N$, let $e_{k}(X_{1},\ldots,X_{N})$ be the
$k$-th elementary symmetric function in the variables $X_{1},\ldots,X_{N}$.
Define the polynomial $\Phi(X)\in\hat{T}_{\delta}[X]$ by
\[
\Phi(X)=\prod_{\gamma\in\Pi_{\delta}}(X+(e^{\gamma},1)).
\]

\begin{proposition}
\label{Prop_DecFI(T)}We have
\begin{equation}
\Phi(X)=\sum_{k=0}^{N}(e_{k}(e^{\gamma_{1}},\ldots,e^{\gamma_{N}}),k)X^{N-k}
\label{dec-Fi(T)}%
\end{equation}
and for any $k=0,\ldots,N,$ the expression $(e_{k}(e^{\gamma_{1}}%
,\ldots,e^{\gamma_{N}}),k)$ decomposes as a sum of elements $(s_{\lambda
},n)\in\hat{T}_{\delta}^{+}$ with positive integer coefficients. In
particular, we have $\Phi(X)\in\hat{T}_{\delta}^{+}[X]$.
\end{proposition}

\begin{proof}
The expression $e_{k}(e^{\gamma_{1}},\ldots,e^{\gamma_{N}})$ is the plethysm
of the elementary symmetric function $e_{k}$ by $s_{\delta}$. This means that
\[
e_{k}(e^{\gamma_{1}},\ldots,e^{\gamma_{N}})=\mathrm{char}\left(
\bigwedge\limits^{k}V(\delta)\right)
\]
is the character of the $k$-th exterior power of the representation
$V(\delta)$.\ Since ${\textstyle\bigwedge\limits^{k}} V(\delta)$ is a
submodule of $V(\delta)^{\otimes k}$, its character indeed decomposes as a sum
of characters in $\{s_{\lambda}\mid V(\lambda)\in V(\delta)^{\otimes k}\}$ with
positive integer coefficients.
\end{proof}

\begin{corollary}
\label{integralClosure} $\hat{T}_{\delta}^{+}$ is integrally closed in
$\hat{T}_{\delta}$.
\end{corollary}

\begin{proof}
Let $\overline{\hat{T}_{\delta}^{+}}$ denote the integral closure of $\hat
{T}_{\delta}^{+}$ in $\hat{T}_{\delta}$. We have $\overline{\hat{T}%
^{+}_{\delta}}\subset\hat{T}_{\delta}$ by definition. Conversely, since
$\overline{\hat{T}^{+}_{\delta}}$ is a ring and $\hat{T}_{\delta}$ is
generated by the monomials $(e^{\gamma},1)$ with $\gamma\in\Pi_{\delta}$, it
suffices to prove that each such $(e^{\gamma},1)$ belongs to $\overline
{\hat{T}_{\delta}^{+}}$.\ But $-(e^{\gamma},1)$ is a root of $\Phi(X)$ which
is, by the previous proposition, a monic polynomial with coefficients in
$\overline{\hat{T}_{\delta}^{+}}$.\ Therefore $-(e^{\gamma},1)$ and
$(e^{\gamma},1)$ are integers over $\hat{T}^{+}_{\delta}$ and thus belong to
$\overline{\hat{T}_{\delta}^{+}}$.
\end{proof}

\section{Minimal boundary of $\Gamma_{\infty}(\mathbb{R}^{d})$}

\subsection{Algebraic description of the growth graph}

Let $\mathcal{G}(\mathbb{R}^{d})$ be the growth graph of $\Gamma
(\mathbb{R}^{d})$ and $\mathcal{G}(\Delta)$ be the one of $\Gamma
(\Delta)$. Namely, the set $\Lambda_{n}$ of vertices of rank $n$ of the
graph $\mathcal{G}(\mathbb{R}^{d})$ are pairs $(\gamma,n)$ where $\gamma$
is a weight of $P_{\delta}$ such that $\Gamma_{\mathbb{R}^{d}}(\gamma
,n)\not =\emptyset$, and the weight of the edge between $(\gamma,n)$ and
$(\gamma^{\prime},n+1)$ is $e((\gamma,n),(\gamma^{\prime},n+1))=\#\Gamma
_{\gamma,\gamma^{\prime}}(\mathbb{R}^{d})$. From the graph embedding of
Section $1$, the set of extreme central measures on $\Gamma_{\infty}(\mathbb{R}^{d})$ is
in bijection with the set of extreme points of the convex set $\partial
\mathcal{H}(\mathcal{G}(\mathbb{R}^{d}))$ of nonnegative functions $p:\coprod_{n\geq
0}\Lambda_{n}\longrightarrow\mathbb{R}^{+}$ with $p(0,0)=1$ and $p(\lambda
,n)=\sum_{\lambda\nearrow\mu}e(\gamma,\mu)p(\mu,n+1)$, and the same holds for
$\mathcal{G}(\Delta)$. An important feature of $\mathcal{G}%
(\mathbb{R}^{d})$ is that this graded graph is multiplicative: it is
related to the algebra $\hat{T}_{\delta}$ as follows.

\begin{proposition}\label{multiplicativeGraphDescription}
$\mathcal{G}(\mathbb{R}^{d})$ is a multiplicative graph associated with the algebra $\hat{P}_{\delta}$ with the injective map
$$i:\left\lbrace \begin{matrix}
\coprod_{n\geq 0}\Lambda_{n}&\longrightarrow&\hat{P}_{\delta}\\
(\gamma,n)&\mapsto& (e^{\gamma},n),n\geq1\\
*&\mapsto& (s_{\delta},1)
\end{matrix}\right.,$$
and $(\hat{P}_{\delta})_{\mathcal{G}(\mathbb{R}^{d})}=\hat{T}_{\delta}$. In particular, $\partial \mathcal{H}_{\infty}(\mathcal{G}(\mathbb{R}^{d}))$ is homeomorphic to $\Mult(\hat{T}_{\delta})^{+}$ through the map
$$i^{*}:\left\lbrace \begin{matrix}
\Mult(\hat{T}_{\delta})^{+}&\longrightarrow&\partial \mathcal{H}_{\infty}(\mathcal{G}(\mathbb{R}^{d}))\\
f&\mapsto& f\circ i\\
\end{matrix}\right.$$
\end{proposition}

\begin{proof}
Since $e((\gamma,n),(\gamma^{\prime},n+1))=\#\Gamma_{\mathbb{R}^{d}}(\gamma,\gamma^{\prime
})=K_{\delta,\gamma^{\prime}-\gamma}$, the following equality holds for
$(\gamma,n)\in\Lambda_{n}$:
\begin{align*}
i(\gamma,n)i(\ast)=  &  (e^{\gamma},n)(\sum_{\kappa\in\Pi_{\delta}}%
K_{\delta,\kappa}e^{\gamma^{\prime}},1)=\sum_{\kappa\in\Pi_{\delta}}%
K_{\delta,\kappa}(e^{\gamma+\kappa},n+1)\\
=  &  \sum_{\substack{\gamma^{\prime}\in P_{\delta},\gamma^{\prime}-\gamma
\in\Pi_{\delta}}}K_{\delta,\gamma^{\prime}-\gamma}(e^{\gamma^{\prime}},n+1)\\
=  &  \sum_{\gamma^{\prime}\in P_{\delta}}e((\gamma,n),(\gamma^{\prime
},n+1))i(\gamma^{\prime},n+1).
\end{align*}
Thus, $\mathcal{G}(\mathbb{R}^{d})$ is a multiplicative graph associated
with $\hat{P}_{\delta}$ through the map $i$. Note that by construction, the
sub-algebra of $\hat{P}_{\delta}$ generated by the elements $\lbrace
i(\gamma,n)\rbrace_{(\gamma,n)\in\mathcal{G}(\mathbb{R}^{d})}$ is
precisely $\hat{T}_{\delta}$: the last part of the proposition is deduced from
Proposition \ref{multiGraphExtreme}.
\end{proof}

\subsection{Characterization of the multiplicative maps on $\hat{T}_{\delta}$}\label{CharacMult}

The set of extreme central measures on $\mathcal{G}(\mathbb{R}^{d})$ is thus
given by the set of positive morphisms from $\hat{T}_{\delta}$ to $\mathbb{R}$
which take the value $1$ on $(s_{\delta},1)$. We will prove in this
subsection the following result:

\begin{proposition}
\label{algebraicDescription} Let $f\in \Mult(\hat{T}_{\delta})^{+}$. There
exists a multiplicative map $\phi:\mathbb{R}[Q^{+}]\longrightarrow
\mathbb{R}^{+}$ and an element $w\in W$ such that
\[
f(e^{\gamma},n)=\frac{1}{\phi(S_{\delta})^{n}}\phi(e^{n\delta-w(\gamma)}),
\]
for all $(e^{\gamma},n)\in\hat{T}_{\delta}$.
\end{proposition}

Note that the element $\phi(e^{n\delta-w(\gamma)})$ is well-defined: indeed,
if $(e^{\gamma},n)\in\hat{T}_{\delta}$, then the weight $\gamma$ appears in
the representation $V(\delta)^{\otimes n}$ and $w(\gamma)$ is thus smaller
than $n\delta$ with respect to the roots order relative to the set of simple
roots $S$. Therefore, $n\delta-w(\gamma)\in Q^{+}$.

Let $f$ be a multiplicative map on $\hat{T}_{\delta}$. Since $f$ is
multiplicative and $\hat{T}_{\delta}$ is generated by the set $\tilde{\Pi
}_{\delta}:=\{(e^{\gamma},1),\gamma\in\Pi_{\delta}\}$, $f$ is completely
determined by its value on $\tilde{\Pi}_{\delta}$. We suppose from now on that
$f\in \Mult(\hat{T}_{\delta})^{+}$. Let $M_{f}=\sum_{\gamma\in\Pi_{\delta}%
}K_{\delta,\gamma}f(\gamma,1)\gamma$: $M_{f}$ belongs to $\mathbb{R}^{d}$,
thus there exists $w\in W$ such that $w(M_{f})\in\Delta$. Replacing $f$ by $f\circ
w^{-1}$ gives another multiplicative map on $\hat{T}_{\delta}$ such that
$M_{f\circ w^{-1}}=\sum_{\gamma\in\Pi_{\delta}}K_{\delta,\gamma}(f\circ
w^{-1})(e^{\gamma},1)\gamma\in\Delta$ and such that $f$ can be expressed from
$f\circ w^{-1}$ with the formula $f=(f\circ w^{-1})\circ w$.

\begin{lemma}
\label{stabilityNonzero} Assume that $M_{f}\in\Delta$ and let $\alpha\in S$. 
For all $\gamma\in\Pi_{\delta}$ such that $\gamma-\alpha\in\Pi_{\delta}$,
\[
f(e^{\gamma},1)=0\Longrightarrow f(e^{\gamma-\alpha},1)=0.
\]
In particular, $f(e^{\delta},1)\not =0$.
\end{lemma}

\begin{proof}
Let $\alpha\in S$, and suppose that there exists $\gamma\in\Pi_{\delta}$ such
that $\gamma-\alpha\in\Pi_{\delta},f(e^{\gamma},1)=0$ and $f(e^{\gamma-\alpha
},1)\not =0$. If $\gamma^{\prime}$ is another vector of $\Pi_{\delta}$ such
that $f(e^{\gamma^{\prime}},1)\not =0$, then necessarily $\gamma^{\prime
}-\alpha\not \in \Pi_{\delta}$: indeed, if $\gamma^{\prime}-\alpha\in
\Pi_{\delta}$, then
\[
f(e^{\gamma^{\prime}-\alpha},1)f(e^{\gamma},1)=f(e^{\gamma+\gamma^{\prime
}-\alpha},2)=f(e^{\gamma-\alpha},1)f(e^{\gamma^{\prime}},1)\not =0,
\]
which contradicts the fact that $f(e^{\gamma},1)=0$. For all
$\gamma^{\prime}\in\Pi_{\delta}$, $\gamma^{\prime}-\alpha\not \in \Pi_{\delta
}$ implies that $\langle\gamma^{\prime},\alpha\rangle\leq0$: thus,
$f(e^{\gamma^{\prime}},1)\not =0$ implies that $\langle\gamma^{\prime}%
,\alpha\rangle\leq0$. We get
\[
\langle\sum_{\gamma^{\prime}\not =\gamma-\alpha}K_{\delta,\gamma^{\prime}
}f(e^{\gamma^{\prime}},1)\gamma^{\prime},\alpha\rangle\leq0.
\]
Since $f(e^{\gamma-\alpha},1)\not =0$ and from the previous argument, we get
$\gamma-2\alpha\not \in \Pi_{\delta}$. Hence, $\frac{2\langle\gamma
,\alpha\rangle}{\langle\alpha,\alpha\rangle}\leq1$, which yields
$\langle\gamma-\alpha,\alpha\rangle<0$. Finally,
\[
\langle M,\alpha\rangle=K_{\delta,\gamma-\alpha}f(e^{\gamma-\alpha}%
,1)\langle\gamma-\alpha,\alpha\rangle+\langle\sum_{\gamma^{\prime}\not =%
\gamma-\alpha}K_{\delta,\gamma^{\prime}}f(e^{\gamma^{\prime}},1)\gamma
^{\prime},\alpha\rangle<0,
\]
which contradicts the fact that $M\in\Delta$. Let $\gamma\in
\Pi_{\delta}$ be such that $f(e^{\gamma},1)\not =0$. Since $\gamma\in\Pi_{\delta
}$, there exists a finite sequence $(x_{i})_{1\leq i\leq r}$ in $S$ such that
$\delta-\sum_{i=1}^{j}x_{i}\in\Pi_{\delta}$ for all $1\leq j\leq r$ and
$\delta-\sum_{i=1}^{r}x_{i}=\gamma$. Thus, from the first part of the lemma,
$f(e^{\delta-\sum_{i=1}^{j}x_{i}},1)\not =0$ for all $1\leq j\leq r$; in
particular, $f(e^{\delta-x_{1}},1)\not =0$, and applying again the first part
of the lemma yields that $f(e^{\delta},1)\not =0$.
\end{proof}

\bigskip

We can now prove Proposition \ref{algebraicDescription}:

\bigskip

\begin{proof}
[Proof of Proposition \ref{algebraicDescription}]Let $f\in
\Mult(\hat{T}_{\delta})^{+}$ be such that $M_{f}\in\Delta$. Let $\alpha\in
S$. If for all $\gamma\in\Pi_{\delta}$ such that $f(e^{\gamma},1)\not =0$ we
have $\gamma-\alpha\not \in \Pi_{\delta}$, then set $\phi(e^{\alpha})=0$.
Otherwise, let $\gamma\in\Pi_{\delta}$ be such that $f(e^{\gamma},1)\not =0$ and
such that $\gamma-\alpha\in\Pi_{\delta}$, and set $\phi(e^{\alpha}%
)=\frac{f(e^{\gamma-\alpha},1)}{f(e^{\gamma},1)}$. Then, $\phi(\alpha)$ is
independent of the choice of $\gamma$. Indeed, if $\gamma^{\prime}$ is another
weight satisfying the same hypothesis, then
\[
f(e^{\gamma},1)f(e^{\gamma^{\prime}-\alpha},1)=f(e^{\gamma+\gamma^{\prime
}-\alpha},2)=f(e^{\gamma-\alpha},1)f(e^{\gamma^{\prime}},1),
\]
so that finally,
\[
\frac{f(e^{\gamma-\alpha},1)}{f(e^{\gamma},1)}=\frac{f(e^{\gamma^{\prime
}-\alpha},1)}{f(e^{\gamma^{\prime}},1)}.
\]
Note that we have in particular proven that for all $\gamma\in\Pi_{\delta}$
such that $\gamma+\alpha\in\Pi_{\delta}$ and $f(e^{\gamma+\alpha},1)\not =0$,
we have
\begin{equation}
\frac{f(e^{\gamma},1)}{f(e^{\gamma+\alpha},1)}=\phi(e^{\alpha}).
\label{crossProduct}%
\end{equation}
Let $\phi:\mathbb{R}[Q^{+}]\longrightarrow\mathbb{R}^{+}$ be the
multiplicative map obtained by extending multiplicatively the map $\phi$
defined on $\{e^{\alpha},\alpha\in S\}$ and by specifying the value
$\phi(1)=1$. Consider the roots order with respect to the set of simple
roots $S$ and let us prove by induction on the roots order that $f(e^{\gamma
},1)=f(e^{\delta},1)\phi(e^{\delta-\gamma})$ for $\gamma\in\Pi_{\delta}$. For
$\gamma=\delta$ the result is straightforward. Let $\gamma\in\Pi_{\delta}$
and suppose that the result is true for all $\gamma^{\prime}>\gamma$. There
exists $\alpha\in S$ such that $\gamma+\alpha\in\Pi_{\delta}$. If
$f(e^{\gamma+\alpha},1)=0$, then $f(e^{\gamma},1)=0$ by Lemma
\ref{stabilityNonzero}; in particular, $f(e^{\gamma},1)=\phi(e^{\alpha
})f(e^{\gamma+\alpha},1)$. By the induction hypothesis, $f(e^{\gamma+\alpha
},1)=f(e^{\delta},1)\phi(e^{\delta-(\gamma+\alpha)})$, and finally,
\[
f(e^{\gamma},1)=\phi(e^{\alpha})f(e^{\delta},1)\phi(e^{\delta-(\gamma+\alpha
)})=f(e^{\delta},1)\phi(e^{\delta-\gamma}).
\]
If $f(e^{\gamma+\alpha},1)\not =0$, then by \eqref{crossProduct} and by the
induction hypothesis,
\[
f(e^{\gamma},1)=\phi(e^{\alpha})f(e^{\gamma+\alpha},1)=\phi(e^{\alpha
})f(e^{\delta},1)\phi(e^{\delta-(\gamma+\alpha)})=f(e^{\delta},1)\phi
(e^{\delta-\gamma}).
\]
Let $(\gamma,n)\in\hat{T}_{\delta}$, and let $\gamma_{1},\dots,\gamma_{n}%
\in\Pi_{\delta}$ such that $\gamma=\sum_{i=1}^{n}\gamma_{i}$. Then, by
multiplicativity of $f$ and the result above, we have
\begin{align*}
f(e^{\gamma},n)=  &  f(e^{\sum_{i=1}^{n}\gamma_{i}},n)=\prod_{i=1}%
^{n}f(e^{\gamma_{i}},1)=\prod_{i=1}^{n}f(e^{\delta},1)\phi(e^{\delta
-\gamma_{i}})\\
=  &  f(e^{\delta},1)^{n}\phi(e^{n\delta-\sum_{i=1}^{n}\gamma_{i}%
})=f(e^{\delta},1)^{n}\phi(e^{n\delta-\gamma})
\end{align*}
Since $f(s_{\delta},1)=1$, we have on the one hand
\[
\sum_{\gamma\in\Pi_{\delta}}K_{\delta,\gamma}f(e^{\gamma},1)=1.
\]
On the other hand, from the result above,
\[
\sum_{\gamma\in\Pi_{\delta}}K_{\delta,\gamma}f(e^{\gamma},1)=\sum_{\gamma
\in\Pi_{\delta}}K_{\delta,\gamma}f(e^{\delta},1)\phi(e^{\delta-\gamma
})=f(e^{\delta},1)\phi(S_{\delta}).
\]
Thus, $f(e^{\delta},1)=\frac{1}{\phi(S_{\delta})}$, which ends the proof of
the proposition in the case $M_{f}\in\Delta$. Suppose that $f$ is a
general nonnegative multiplicative function on $\hat{T}_{\delta}$. Let $w\in W$ be such
that $M_{f\circ w^{-1}}\in\Delta$. By the first part of the proof, there
exists $\phi\in \Mult(\mathbb{R}[Q+])^{+}$ such that $f\circ w^{-1}(e^{\gamma
},n)=\frac{1}{\phi(S_{\delta})^{n}}\phi(e^{n\delta-\gamma})$. Thus, composing
$f\circ w^{-1}$ with $w$ yields that $f(e^{\gamma},n)=\frac{1}{\phi(S_{\delta
})^{n}}\phi(e^{n\delta-w(\gamma)})$ for $(\gamma,n)\in\hat{T}_{\delta}$.
\end{proof}

\begin{remark}
Suppose that $\phi(e^{\alpha})\not =0$ for all $\alpha\in S$. Then, the map
$\phi$ extends to a homomorphism $\phi:\mathbb{R}[P]\longrightarrow
\mathbb{R}^{+}$ with the formula
\[
\phi(e^{\gamma})=\prod_{\alpha\in S}\phi(e^{\alpha})^{r_{\alpha}}\text{ for
}\gamma=\sum_{\alpha\in S}r_{\alpha}\alpha.
\]
In this case,
\[
f(e^{\gamma},n)=f(e^{\delta},1)^{n}\phi(e^{\lambda-n\delta})=\left(
\frac{f(e^{\delta},1)}{\phi(e^{\delta})}\right)  ^{n}\phi(e^{\gamma}).
\]
Since, $f(s_{\delta},1)=1$, $\frac{f(e^{\delta},1)}{\phi(e^{\delta})}%
=\phi(s_{\delta})^{-1}$. Hence, when $\phi(e^{\alpha})>0$ for all
$\alpha\in S$, $f$ can be written on $\hat{T}_{\delta}$ as
\[
f(e^{\gamma},n)=\frac{\phi(e^{\gamma})}{\phi(s_{\delta})^{n}},
\]
with $\phi:P\longrightarrow\mathbb{R}^{+}$ a multiplicative map.\\
In this case, for $(s_{\lambda},n)\in\hat{T}_{\delta}^{+}$, we have also
\[
f(s_{\lambda},n)=\frac{\phi(s_{\lambda})}{\phi(s_{\delta})^{n}},
\]
where the restriction $\phi:\mathcal{A}_{\delta}\longrightarrow \mathbb{R}^{+}$ is again a multiplicative map. 
\end{remark}

To summarize, let us define the map $\Phi:\Mult(\mathbb{R}[Q^{+}])^{+}\times
W\longrightarrow \Mult(\hat{T}_{\delta})^{+}$ by
\[
\Phi(\phi,w)(e^{\gamma},n)=\frac{1}{\phi(S_{\delta})}\phi(e^{n\delta
-w(\gamma)}).
\]
Proposition \ref{algebraicDescription} yields that the map $\Phi$ is
surjective. Since $\mathbb{R}[Q^{+}]$ is the free commutative algebra
generated by $\{e^{\alpha},\alpha\in S\}$, $\Mult(\mathbb{R}[Q^{+}])^{+}$ is
isomorphic to $(\mathbb{R}^{+})^{d}$ through the map $\theta:\Mult(\mathbb{R}%
[Q^{+}])^{+}\longrightarrow(\mathbb{R}^{+})^{d}$ given by $\theta(\phi
)=(\phi(e^{\alpha_{i}}))_{1\leq i\leq d}$ for $\phi\in \Mult(\mathbb{R}%
[Q^{+}])^{+}$. The composition of $\Phi$ with $\theta^{-1}$ yields thus a
surjective map $(\mathbb{R}^{+})^{d}\times W\longrightarrow \Mult(\hat
{T}_{\delta})^{+}$. Since $\Phi$ is not
necessarily injective, the latter map is not bijective. The lack of
injectivity comes from two facts: first, if $M_{f}$ lies at the intersection
of two Weyl chambers, then $M_{f\circ w^{-1}}\in\Delta$ for several $w\in W$.
Secondly, some degeneracy may occur when $\delta$ is orthogonal to some
simple roots. The goal of the next subsection is to overcome the second problem.

\subsection{Dominant faces of the weight polytope}\label{Vinberg}

Let $f\in \Mult(\hat{T}_{\delta})$ such that $M_{f}\in\Delta$; it is possible
to give a geometric description of the set $\Pi_{\delta}(f):=\{\gamma\in
\Pi_{\delta}\mid f(e^{\gamma},1)\not =0)\}$. A dominant face $F$ is a
face of the polytope $K(\delta)$ such that $F\cap\Delta\not =0$. We denote by
$\Pi_{F}$ the intersection of $\Pi_{\delta}$ with $F$. We say that a subset
$S^{\prime}\subset S$ of simple roots is $\delta$-admissible if each
indecomposable component of $S^{\prime}$ contains a root which is not
orthogonal to $\delta$; in particular, according to this definition, the empty
set is a $\delta$-admissible subset, since it has no indecomposable component.
For each subset $S^{\prime}\subset S$, denote by $W_{S^{\prime}}$ the Weyl
group generated by the elements $s_{\alpha^{\prime}},\alpha^{\prime}\in
S^{\prime}$ (where $W_{\emptyset}$ is simply $\lbrace \Id\rbrace$). We will use
the following results which comes from \cite{Vin}.

\begin{theorem}
\label{Description faces} Assigning to each $\delta$-admissible subset
$S^{\prime}\subset S$ the polytope $F_{S^{\prime}}=\Conv(w^{\prime
}\delta\mid w^{\prime}\in W_{S^{\prime}})$ yields a one-to-one correspondence
between $\delta$-admissible subsets of $S$ and dominant faces of the polytope
$K(\delta)$. Moreover, the set $\Pi_{F_{S^{\prime}}}$ coincides with
the set $(\delta+\langle S^{\prime}\rangle)\cap\Pi_{\delta}$ and $\dim
F_{S^{\prime}}=\#S^{\prime}$
\end{theorem}

This yields the following characterization $\Pi_{\delta}(f)$.

\begin{proposition}
\label{descriptionNonZeroSet} There exists a dominant face $F$ of the weight
polytope $K(\delta)$ such that $\Pi_{\delta}(f)=\Pi_{F}$.
\end{proposition}

Before proving Proposition \ref{descriptionNonZeroSet}, let us prove the
following lemma:

\begin{lemma}
\label{admissible-nonzerodecomp} Let $S^{\prime}\subset S$ and $\gamma\in
\Pi_{\delta}$ such that $\delta-\gamma=\sum_{\alpha\in S^{\prime}}k_{\alpha
}\alpha$ with $k_{\alpha}>0$ for all $\alpha\in S^{\prime}$. Then, $S^{\prime
}$ is $\delta$-admissible and $\gamma\in F_{S^{\prime}}$.
\end{lemma}

The proof of this lemma uses ingredients similar to those of Vinberg in
\cite{Vin}.

\begin{proof}
Suppose $\gamma$ can be written as
\[
\gamma=\delta-\sum_{\alpha\in S^{\prime}}k_{\alpha}\alpha,
\]
with $S^{\prime}$ a subset of $S$ and $k_{\alpha}\in\mathbb{N}^{\ast}$ for
$\alpha\in S^{\prime}$. Since $\gamma\in\Pi_{\delta}$, there exists a sequence
$(\gamma_{i})_{0\leq i\leq t}$ with $t=\sum_{\alpha\in S^{\prime}}k_{\alpha}$
such that $\gamma_{i}\in\Pi_{\delta}$, $\gamma_{0}=\gamma$, $\gamma_{t}%
=\delta$ and $\gamma_{i+1}-\gamma_{i}\in S$. Since for all $\gamma\in
\Pi_{\delta}$, $\delta-\gamma$ is a sum of simple roots with nonnegative
coefficients, for all $0\leq i\leq t-1$ we have $\gamma_{i+1}-\gamma_{i}\in
S^{\prime}$ and $\#\{0\leq i\leq t-1|\gamma_{i+1}-\gamma_{i}=\alpha
\}=k_{\alpha}$ for $\alpha\in S^{\prime}$. This implies in particular that
$\gamma_{i}\in\delta+\langle S^{\prime}\rangle$ for all $0\leq i\leq t$. Let
$\alpha\in S^{\prime}$: since $k_{\alpha}>0$, there exists $0\leq i_{\alpha
}\leq t-1$ such that $\gamma_{i+1}-\gamma_{i}=\alpha$. This yields that
$\dim(K(\delta)\cap(\delta+\langle S^{\prime}\rangle)=\#S^{\prime}$. Let $l$
be the linear form such that $l(\alpha)=1$ for $\alpha\in S\setminus S^{\prime}$
and $l(\alpha)=0$ for $\alpha\in S$. For $\gamma\in\Pi_{\delta}$,
$\delta-\gamma$ is a sum of simple roots with positive coefficients, thus
$l(\gamma)\leq l(\delta)$, with equality if and only if $\gamma\in
\delta+\langle S^{\prime}\rangle$. Thus, $(K(\delta)\cap(\delta+\langle
S^{\prime}\rangle)$ is a face of the polytope $K_{\delta}$; since
$\dim(K(\delta)\cap(\delta+\langle S^{\prime}\rangle)=\#S^{\prime}$,
$S^{\prime}$ is $\delta$-admissible by \cite[p.10]{Vin}. Finally,
$K(\delta)\cap(\delta+\langle S^{\prime}\rangle)=F_{S^{\prime}}$ and
$\gamma\in F_{S^{\prime}}$.
\end{proof}

\begin{lemma}
\label{faceNonzeroRootsNonZero} $\Pi_{F}\subset\Pi_{\delta}(f)$ if and only if
$\phi$ is nonzero on $S_{F}$.
\end{lemma}

\begin{proof}
Suppose that $\Pi_{F}\subset\Pi_{\delta}(f)$. Let $\alpha_{0}\in S_{F}$. Since
$f(e^{\gamma},1)$ is nonzero for $\gamma\in\Pi_{F}$, by Lemma
\ref{stabilityNonzero} and the definition of $\phi$ it suffices to prove that
there exists $\gamma\in\Pi_{F}$ such that $\gamma+\alpha_{0}\in\Pi_{F}$ or
$\gamma-\alpha_{0}\in\Pi_{F}$. Since $S_{F}$ is $\delta$-admissible, $\dim F=\#
S_{F}$; $F=\Conv(w.\delta\mid w\in W_{S_{F}})$ and $\dim F=\# S_{F}$, thus
there exists $w\in W_{F}$ such that $\delta-w.\delta=\sum_{\alpha\in S_{F}%
}k_{\alpha}\alpha$ with $k_{\alpha_{0}}>0$. This implies the existence of
$\gamma\in\Pi_{F}$ such that $\gamma+\alpha_{0}\in\Pi_{F}$. Since $\Pi
_{F}\subset\Pi_{\delta}(f)$, $f(e^{\gamma+\alpha_{0}},1)\not =0$ and
$f(e^{\gamma},1)\not =0$, and thus
\[
\phi(\alpha_{0})=\frac{f(e^{\gamma},1)}{f(e^{\gamma+\alpha_{0}},1)}\not =0.
\]
Conversely, suppose that $\phi$ is nonzero on $S_{F}$. By Theorem
\ref{Description faces}, $\Pi_{F}=(\delta+\langle S_{F}\rangle)\cap\Pi
_{\delta}$. Since $f(e^{\delta},1)\not =0$ and $\phi$ is nonzero on $S_{F}$,
$f$ is nonzero $\Pi_{F}$ by Proposition \ref{algebraicDescription}.
\end{proof}

\bigskip

We turn now to the proof of Proposition \ref{descriptionNonZeroSet}.

\bigskip

\begin{proof}
We order the set of dominant faces by the inclusion order; note that the set
of dominant faces is a lattice with respect to this order, and we denote by
$F\vee F^{\prime}$ the supremum of two dominant faces $F$ and $F^{\prime}$:
$F\vee F^{\prime}$ is the smallest dominant face containing both $F$ and
$F^{\prime}$. Let $\gamma\in\Pi_{\delta}$ such that $f(e^{\gamma},1)\not =0$,
and let $F$ be the smallest dominant face containing $\gamma$. We denote by
$S_{F}$ the $\delta$-admissible subset of simple roots corresponding to $F$.
Then, $\gamma$ can be written as
\[
\gamma=\delta-\sum_{\alpha\in S_{F}}k_{\alpha}\alpha
\]
with $k_{\alpha}\in\mathbb{N}$. Necessarily, we have $k_{\alpha}>0$ for all
$\alpha\in S_{F}$. Otherwise, Lemma \ref{admissible-nonzerodecomp} would imply
that $\gamma$ belongs to a smaller dominant face of $K(\delta)$. Let
$(\gamma_{i})_{0\leq i\leq t}$ with $t=\sum_{\alpha\in S^{\prime}}k_{\alpha}$
be a sequence of $\Pi_{\delta}$ such that $\gamma_{i}\in\Pi_{\delta}$,
$\gamma_{0}=\gamma$, $\gamma_{t}=\delta$ and $\gamma_{i+1}-\gamma_{i}\in
S_{F}$. Since $f(e^{\gamma_{0}},1)\not =0$, Lemma \ref{stabilityNonzero}
yields that $f(e^{\gamma_{i}},1)\not =0$ for $0\leq i\leq t$. Let $\alpha\in
S_{F}$: since $k_{\alpha}>0$, a similar deduction as in the proof of the
previous lemma yields that there exists $0\leq i\leq t-1$ such that
$\gamma_{i+1}-\gamma_{i}=\alpha$. Therefore,
\[
\phi(e^{\alpha})=\frac{f(e^{\gamma_{i}},1)}{f(e^{\gamma_{i+1}},1)}\not =0.
\]
Since $\phi(e^{\alpha})\not =0$ for $\alpha\in S_{F}$, $f(e^{\gamma},1)$ is
nonzero on $\Pi_{\delta}\cap(\delta-\langle S_{F}\rangle)$, and $\Pi
_{F}\subset\Pi_{\delta}(f)$. We have thus proven that if a weight $\gamma$ is
in $\Pi_{\delta}(f)$, then the intersection of $\Pi_{\delta}$ with the
smallest dominant face containing $\gamma$ is also included in $\Pi_{\delta
}(f)$; hence, $\Pi_{\delta}(f)$ is an union of sets $\Pi_{F}$, where $F$ are
dominant faces. 

Let $F$ and $F^{\prime}$ be two dominant faces such
that $\Pi_{F},\Pi_{F^{\prime}}\subset\Pi_{\delta}(f)$, and let us show that
$\Pi_{F\vee F^{\prime}}\subset\Pi_{\delta}(f)$. Note first that
$F\vee F^{\prime}=F_{S_{F}\cup S_{F^{\prime}}}$: on the first hand, the
smallest vector space containing both $\langle S_{F}\rangle$ and $\langle
S_{F^{\prime}}\rangle$ is $\langle S_{F}\cup S_{F^{\prime}}\rangle$. On the
other hand, since $S_{F}$ and $S_{F^{\prime}}$ are $\delta$-admissible,
$S_{F}\cup S_{F^{\prime}}$ is again $\delta$-admissible. It suffices thus to
show that $\Pi_{F_{S_{F}\cup S_{F^{\prime}}}}\subset\Pi_{\delta}(f)$. But
Lemma \ref{faceNonzeroRootsNonZero} yields that $\phi(e^{\alpha})$ is nonzero
for $\alpha\in S_{F}$ and $\alpha\in S_{F^{\prime}}$. Thus, $\phi(e^{\alpha})$
is nonzero for $\alpha\in S_{F}\cup S_{F^{\prime}}$ and $\Pi_{F_{S_{F}\cup
S_{F^{\prime}}}}\subset\Pi_{\delta}(f)$. Let $F_{0}$ be the supremum of
$\{F\text{ dominant face of }K(\delta),\Pi_{F}\subset\Pi_{\delta}(f)\}$. By
the previous argument, $\Pi_{F_{0}}\subset\Pi_{\delta}(f)$. Let $\gamma\in
\Pi_{\delta}(f)$ and let $F$ be the smallest dominant face of $K(\delta)$
containing $\gamma$. By the first part of the proof, $\Pi_{F}\subset
\Pi_{\delta}(f)$. Thus $F\subset F_{0}$ and $\gamma\in F_{0}$: this proves
that $\Pi_{\delta}(f)\subset\Pi_{F_{0}}$, and finally $\Pi_{\delta}%
(f)=\Pi_{F_{0}}$.
\end{proof}

\begin{corollary}
\label{unicity} Let $f\in \Mult(\hat{T}_{\delta})^{+}$ be such that $M_{f}%
\in\Delta$. There exists a unique $\phi\in \Mult(\mathbb{R}[Q^{+}])^{+}$ such
that $\Phi(\phi,\Id)=f$ and $\{\alpha,\phi(e^{\alpha})\not =0\}$ is a $\delta
$-admissible subset of $S$.
\end{corollary}

\begin{proof}
Let $\phi$ be such that $\Phi(\phi,\Id)=f$. By Proposition
\ref{descriptionNonZeroSet}, there exists a face $F$ of $K(\delta)$ such that
$\Pi_{F}=\Pi_{\delta}(f)$. Lemma \ref{faceNonzeroRootsNonZero} yields that
$\phi$ is nonzero on $S_{F}$. Let $\alpha\in S_{F}$: then, there exists
$\gamma\in\Pi_{\delta}$ such that $\gamma\in\Pi_{F},\gamma-\alpha\in\Pi_{F}$;
thus, $f(e^{\gamma},1)\not =0$ and $f(e^{\gamma-\alpha},1)\not =0$. Therefore,
the value of $\phi$ on $\alpha$ has to be equal to $\frac{f(e^{\gamma-\alpha
},1)}{f(e^{\gamma},1)}$. Hence, there exists at most one $\phi$ such that
$\{\alpha,\phi(e^{\alpha})\not =0\}$ is the $\delta$-admissible subset $S_{F}%
$. Such a map $\phi$ exists, since $f$ is zero on $\Pi_{\delta}\setminus
\Pi_{\delta}(f)$. Suppose that there exists a bigger $\delta
$-admissible subset $S_{F}\subsetneq S^{\prime}$ such that $\phi$ is nonzero
on $S^{\prime}$. Then by Lemma \ref{faceNonzeroRootsNonZero}, $\Pi
_{F_{S^{\prime}}}\subset\Pi_{\delta}(f)$. But by Theorem
\ref{Description faces}, there is a bijection between dominant faces and
$\delta$-admissible subsets: therefore, $\Pi_{\delta}(f)=\Pi_{F}\subsetneq
\Pi_{F_{S^{\prime}}}\subset\Pi_{\delta}(f)$, which is a contradiction. Thus,
there exists exactly one map $\phi$ such that $\Phi(\phi,\Id)=f$ and
$\{\alpha\in S,\phi(e^{\alpha})\not =0\}$ is a $\delta$-admissible subset (and
this $\delta$-admissible subset has to be $S_{F}$).
\end{proof}

\subsection{Identification of the minimal boundary}

We give in this subsection a complete description of the minimal boundary by
describing $\Mult(\hat{T}_{\delta})^{+}$.

\begin{lemma}
\label{Prop_reduc[0,1]}Let $f\in \Mult(\hat{T}_{\delta})^{+}$ be such that
$M_{f}\in\Delta$, and let $\phi\in \Mult(\mathbb{R}[Q^{+}])^{+}$ be such that
$\Phi(\phi,\Id)=f$. Then $\phi(e^{\alpha})\in[0,1]$.
\end{lemma}

\begin{proof}
Let $f\in \Mult(\hat{T}_{\delta})^{+}$ be such that
\[
M=\sum_{\gamma\in\Pi_{\delta}}K_{\delta,\gamma}f(e^\gamma,1)\gamma\in\Delta.
\]
Let $\phi\in \Mult(\mathbb{R}[Q^{+}])^{+}$ be a morphism associated with $f$ by
Proposition \ref{algebraicDescription}, and let $\alpha\in S$. Since
$M\in\Delta$, $\langle M,\alpha\rangle\geq0$. Moreover,
\[
\langle M,\alpha\rangle=\sum_{\gamma\in\Pi_{\delta}}K_{\delta,\gamma}f(e^{\delta}%
,1)\phi(e^{\delta-\gamma})\langle\gamma,\alpha\rangle.
\]
By invariance of $\Pi_{\delta}$ under the symmetry $s_{\alpha}$, $\gamma\in
\Pi_{\delta}$ implies that $s_{\alpha}(\gamma)\in\Pi_{\delta}$. Since
$s_{\alpha}^{2}=\Id$ and since $s_{\alpha}(\gamma)=\gamma$ if and only
if$\langle\alpha,\gamma\rangle=0$, we have
\begin{align*}
\langle M,\alpha\rangle=  &  f(e^{\delta},1)\sum_{\substack{\gamma\in
\Pi_{\delta}\\\langle\gamma,\alpha\rangle>0}}K_{\delta,\gamma}(\phi(e^{\delta-\gamma}%
)-\phi(e^{\delta-s_{\alpha}(\gamma)}))\langle\gamma,\alpha\rangle\\
=  &  f(e^{\delta},1)\sum_{\substack{\gamma\in\Pi_{\delta}\\\langle
\gamma,\alpha\rangle>0}}K_{\delta,\gamma}\phi(e^{\delta-\gamma})(1-\phi(e^{\alpha}%
)^{\frac{2\langle\gamma,\alpha\rangle}{\langle\alpha,\alpha\rangle}}%
)\langle\gamma,\alpha\rangle.
\end{align*}
If $\phi(e^{\alpha})>1$, then $(1-\phi(e^{\alpha})^{\frac{2\langle
\gamma,\alpha\rangle}{\langle\alpha,\alpha\rangle}})<0$ for all $\gamma\in
\Pi_{\delta}$ such that $\langle\gamma,\alpha\rangle>0$, and thus $\langle
M,\alpha\rangle<0$: this would contradict the choice of $f$. Therefore,
$\phi(e^{\alpha})\leq1$.
\end{proof}

The set $\{1,\dots,d\}$ is identified with $S$ by ordering the set of
simple roots, and for $S^{\prime}\subset S$, we denote by $W^{S^{\prime}}$ the
set of minimal right-coset representatives with respect to $S'$: namely,
\[
W^{S^{\prime}}=\{w\in W|l(sw)>l(w)\text{ for }s\in S^{\prime}\}.
\]
For $x\in\Delta$, we denote by $S_{x}$ the set $\{\alpha\in S,\langle
\alpha,x\rangle=0\}$.

\begin{lemma}
\label{unicityRightCoset} Let $x\in\mathbb{R}^{d}$ and let $y$ be the unique
element of $Wx$ belonging to $\Delta$. There exists a unique element $w\in
W^{S_{y}}$ such that $wx=y$.
\end{lemma}

\begin{proof}
Let $W_{y}$ be the parabolic subgroup generated by $S_{y}$. Then, $W_{y}$ is
the stabilizer of $y$. In particular, the set $\lbrace w\in W,w(y)=x\rbrace$ is a left coset of
$W_{y}$ in $W$, and thus the set $\lbrace w\in W,w(x)=y\rbrace$ is a right
coset of $W_{y}$ in $W$. By \cite[1.10]{Hum}, there exists a unique $\tilde
{w}\in W^{S_{y}}$ such that $\lbrace w\in W,w(x)=y\rbrace\tilde{w}wW_{y}$. Thus, there
exists a unique $\tilde{w}\in W^{S_{y}}$ such that $\tilde{w}(x)=y$.
\end{proof}

For each $d$-tuple $\vec{t}=(t_{1},\dots,t_{d})$, denote by
$\mathbf{0}^{c}(\vec{t})$ the set of indices $i$ such that $t_{i}\not =0$ and
by $\mathbf{1}(\vec{t})$ the set of indices $i$ such that $t_{i}=1$. Now
consider the set $[0,1]_{\delta}^{d}$ such that
\begin{equation}
\lbrack0,1]_{\delta}^{d}:=\{\vec{t}\in\lbrack0,1]^{d}\mid\mathbf{0}^{c}(\vec{t})\text{ is $\delta$-admissible}\}. \label{def[0,1]delta}%
\end{equation}
This set will turn out to be a natural parametrization of $K(\delta)^{+}$.
Then, we will prove in Section \ref{SectionMean} that there exists a natural map
$t:K(\delta)\longrightarrow[0,1]_{\delta}^{d}\times W$, written as
$t(m)=(\vec{t}_{m},w_{m})$, such that $t(K(\delta)^{+})=[0,1]_{\delta}^{d}\times
\Id$.

\begin{proposition}
\label{isomosphismeMult} The map $\Phi\circ\theta^{-1}$ yields a bijection
$\Psi$ between $\Mult(\hat{T}_{\delta})^{+}$ and
\[
\lbrace(\vec{t},w)\in[0,1]_{\delta}^{d}\times W\vert w\in W^{\mathbf{1}(\vec{t}%
)}\rbrace.
\]

\end{proposition}

\begin{proof}
Let $f\in \Mult(\hat{T}_{\delta})^{+}$. Let $y$ be the unique point in $W(M_{f})\cap\Delta$ and denote
by $S_{y}$ the set $\{\alpha\in S\mid\langle\alpha,y\rangle=0\}$. By Lemma
\ref{unicityRightCoset}, there exists a unique $w\in W^{S_{y}}$ such that
$w(M_{f})=y$. Thus, by Proposition \ref{algebraicDescription}, Corollary
\ref{unicity} and Lemma \ref{Prop_reduc[0,1]}, there exists a unique $\phi\in
\Mult(\mathbb{R}[Q^{+}])^{+}$ such that $\Phi(\phi,w)=f$ and $\{\alpha\in
S\mid\phi(e^{\alpha})\neq0\}$ is a $\delta$-admissible subset. In order to
conclude, we just have to show that $\phi(e^{\alpha})=1$ if and only if
$\langle\alpha,w(M_{f})\rangle=0$: but, as in the proof of Lemma
\ref{Prop_reduc[0,1]}, we have
\begin{align*}
\langle\alpha,w(M_{f})\rangle=  &  \langle w^{-1}(\alpha),M_{f}\rangle
=\sum_{\gamma\in\Pi_{\delta}}K_{\delta,\gamma}f(e^{\gamma},1)\langle
w^{-1}(\alpha),\gamma\rangle\\
=  &  \sum_{\gamma\in\Pi_{\delta}}K_{\delta,\gamma}\frac{1}{\phi(S_{\delta}%
)}\phi(e^{\delta-w(\gamma)})\langle w^{-1}(\alpha),\gamma\rangle=\sum
_{\gamma\in\Pi_{\delta}}K_{\delta,\gamma}\frac{1}{\phi(S_{\delta})}%
\phi(e^{\delta-w(\gamma)})\langle\alpha,w(\gamma)\rangle\\
=  &  \sum_{\gamma\in\Pi_{\delta}}K_{\delta,\gamma}\frac{1}{\phi(S_{\delta}%
)}\phi(e^{\delta-\gamma})\langle\alpha,\gamma\rangle=\frac{1}{\phi(S_{\delta
})}\sum_{\substack{\gamma\in\Pi_{\delta}\\\langle\gamma,\alpha\rangle>0}%
}K_{\delta,\gamma}(\phi(e^{\delta-\gamma})-\phi(e^{\delta-s_{\alpha}(\gamma)}))\langle
\alpha,\gamma\rangle\\
=  &  \frac{1}{\phi(S_{\delta})}\sum_{\substack{\gamma\in\Pi_{\delta}%
\\\langle\gamma,\alpha\rangle>0}}K_{\delta,\gamma}\phi(e^{\delta-\gamma})(1-\phi(e^{\alpha
})^{\frac{2\langle\gamma,\alpha\rangle}{\langle\alpha,\alpha\rangle}}%
)\langle\alpha,\gamma\rangle,\\
&
\end{align*}
where the fourth inequality is due to the fact that $w$ yields a bijection on
the set of weights which satisfies the relation $K_{\delta,w(\gamma)}=K_{\delta,\gamma}%
$ for each $\gamma\in\Pi_{\delta}$. Thus, $\langle M_{f},\alpha\rangle=0$ if and only $\phi(e^{\alpha})=1$.
\end{proof}

Note that the bijection $\Psi$ in the above proposition is explicitly
given by Proposition \ref{algebraicDescription}: for $\vec{t}\in
[0,1]_{\delta}^{d}$, denote by $\phi_{\vec{t}}$ the unique element of
$\Mult(\mathbb{R}[Q^{+}])^{+}$ such that $\{\phi_{t}(\alpha)\not =0\}$ is
$\delta$- admissible and $\Phi(\phi_{t},w)=\Psi(\vec{t},w)$. Then,
\[
\Psi(\vec{t},w)(e^{\gamma},n)=\frac{1}{\phi_{\vec{t}}(S_{\delta})^{n}}%
\phi_{\vec{t}}(e^{n\delta-w(\gamma}))=\frac{\vec{t}^{n\delta-w(\gamma)}%
}{S_{\delta}(\vec{t})^{n}},
\]
for $(e^{\gamma},n)\in\hat{T}_{\delta}$.

\begin{remark}
The restriction of the set of parameters $(\vec{t},w)$ from $[0,1]^{d}\times
W$ to $\{(\vec{t},w)\in\lbrack0,1]_{\delta}^{d}\times W|w\in W^{\mathbf{1}(\vec
{t})}\}$ is only useful to ensure the injectivity of the map $\Psi$. It is
however still possible to define an element of $\Mult(\hat{T}_{\delta})^{+}$ by
applying the map $\Phi\circ\theta^{-1}$ to any element $(\vec{t},w)$. 
The lack of injectivity without the restriction of the parameters can be seen
in the following example: consider the Lie algebra of type $A_{2}$ with set of simple
roots $\{\alpha_{1},\alpha_{2}\}$ and choose $\delta=\omega_{1},$ the first fundamental weight. Then, $(\omega_{1},\alpha_{2})=0$, and thus any
weight $\gamma\neq\omega_{1}$ of $\Pi_{\omega_{1}}$ is written
$\gamma=\omega_{1}-k_{1}\alpha_{1}-k_{2}\alpha_{2}$ with $k_{1}>0$: hence, if
$t_{1}=0$, we have $\phi(e^{\omega_{1}-\gamma})=\delta_{\gamma,\omega_{1}}$
for all value of $t_{2}$. On the other hand, the $\omega_{1}$-admissible
subsets of $\{\alpha_{1},\alpha_{2}\}$ are $\emptyset$, $\{\alpha_{1}\}$ and
$\{\alpha_{1},\alpha_{2}\}$. Thus the empty $\omega_{1}$-admissible subset
$\emptyset$ yields the unique choice of $t_{2}$ such that $t_{1}=0$ and
$\mathbf{0}^{c}(\vec{t})$ is $\omega_{1}$-admissible, namely $t_{2}=0$.\ The
latter procedure has singled out a particular choice of parameters
$t_{1}=0,t_{2}=0$ among all the choices of $\vec{t}$ yielding the map
$\phi(e^{\omega_{1}-\gamma})=\delta_{\gamma,\omega_{1}}$.
\end{remark}

A straightforward application of Proposition
\ref{multiplicativeGraphDescription} yields the following corollary:

\begin{corollary}
\label{algebraDescripMinimalBoundaryWeight} The map $i\circ\Psi$ gives a
bijection between the minimal boundary $\partial\mathcal{H}_{\infty
}(\mathbb{R}^{d})$ and
\[
\{(\vec{t},w)\in\lbrack0,1]_{\delta}^{d}\times W|w\in W^{\mathbf{1}(\vec{t})}\}.
\]

\end{corollary}

\section{Minimal boundary of $\Gamma_{\infty}(\Delta)$}

In this section, we use the description of $\partial\mathcal{H}_{\infty
}(\mathbb{R}^{d})$ to get the one of $\partial\mathcal{H}_{\infty}(\Delta)$.

\subsection{Algebraic description of the growth graph of $\Gamma_{\infty}(\Delta)$}

The growth graph $\mathcal{G}(\Delta)$ of $\Gamma_{\infty}(\Delta)$ admits a
similar description as the one of $\Gamma_{\infty}(\mathbb{R}^{d})$. The set
$\Lambda_{n}^{+}$ of vertices of rank $n$ of the graph $\mathcal{G}(\Delta)$ are pairs $(\lambda,n)$ where $\lambda$ is a weight of $P_{\delta}^{+}$
such that $\Gamma_{\Delta}(\lambda,n)\not =\emptyset$, and the weight of the
edge between $(\lambda,n)$ and $(\mu,n+1)$ is $e^{+}((\lambda,n),(\mu
,n+1))=\#\Gamma_{\Delta}(\lambda,\mu)$. Moreover, we have the following
algebraic description of $\mathcal{G}(\Delta)$:

\begin{proposition}
$\mathcal{G}(\Delta)$ is a multiplicative graph associated with the algebra $\hat{\mathcal{A}}_{\delta}$ with the injective map
$$i:\left\lbrace \begin{matrix}
\coprod_{n\geq 0}\Lambda_{n}&\longrightarrow&\hat{\mathcal{A}}_{\delta}\\
(\lambda,n)&\mapsto& (s_{\lambda},n),n\geq1\\
*&\mapsto& (s_{\delta},1)
\end{matrix}\right.,$$
and $\left(\hat{\mathcal{A}}_{\delta}\right)_{\mathcal{G}(\Delta)}=\hat{T}^{+}_{\delta}$. In particular, $\partial \mathcal{H}(\mathcal{G}(\Delta))$ is isomorphic to $\Mult(\hat{T}_{\delta}^{+})^{+}$ through the map
$$i^{*}:\left\lbrace \begin{matrix}
\Mult(\hat{T}^{+}_{\delta})^{+}&\longrightarrow&\partial \mathcal{H}(\mathcal{G}(\Delta))\\
f\mapsto f\circ i\\
\end{matrix}\right.$$
\end{proposition}

\begin{proof}
By Littelmann's paths theory, the following equality holds for $(\lambda,n)\in\Lambda_{n}$:
\begin{align*}
i(\lambda,n)i(*)=  &  (s_{\lambda},n)(s_{\delta},1)=\sum_{\mu\in
\mathcal{A}_{\delta}}\#\Gamma_{\Delta}(\lambda,\mu)(s_{\mu},n+1)\\
=  &  \sum_{\mu\in\mathcal{A}_{\delta}}e\big((\lambda,n),(\mu,n+1)\big)i(\mu,n+1).
\end{align*}
Thus, $\mathcal{G}(\Delta)$ is a multiplicative graph associated with
$\hat{\mathcal{A}}_{\delta}$ with the map $i$. Note that by construction,
$\left(  \hat{\mathcal{A}}_{\delta}\right)  _{\mathcal{G}(\Delta)}=\hat
{T}_{\delta}^{+}$: the last part of the proposition is deduced from
Proposition \ref{multiGraphExtreme}.
\end{proof}

Now we are going to connect the sets $\Mult(\hat{T}^{+}_{\delta
})^{+}$ and $\Mult(\hat{T}_{\delta})^{+}$.

\subsection{Relation between $\Mult(\hat{T}_{\delta})^{+}$ and $\Mult(\hat
{T}_{\delta}^{+})^{+}$}

Recall that $\hat{T}_{\delta}^{+}$ is a subalgebra of $\hat{T}_{\delta}$;
therefore, any nonnegative morphism on $\hat{T}_{\delta}$ induces by
restriction a nonnegative morphism on $\hat{T}_{\delta}^{+}$. This yields a
map $\chi:\Mult(\hat{T}_{\delta})^{+}\longrightarrow \Mult(\hat{T}_{\delta}%
^{+})^{+}$. The important step in the description of $\Mult(\hat{T}_{\delta
}^{+}))^{+}$ is the following:

\begin{proposition}
\label{multiRelationWeightCharacter} The map $\chi$ yields a surjection from
$\{f\in \Mult(\hat{T}_{\delta})^{+}\mid M_{f}\in\Delta\}$ to $\Mult(\hat
{T}_{\delta}^{+})^{+}$.
\end{proposition}

The proof of this proposition needs some preparation. Let $f$ be a nonnegative morphism from $\hat{T}_{\delta}^{+}$ to $\mathbb{R}%
$. By Corollary \ref{integralClosure} and by Corollary 4 page 35 in \cite{BBK},
$f$ can be extended to a morphism $\tilde{f}$ from $\hat{T}_{\delta}$ to
$\mathbb{C}$. The first task is to prove that $\tilde{f}\in \Mult(\hat
{T}_{\delta})^{+}$.

We need to recall a classical result by Aissen, Edrei, Schoenberg and White on
polynomials with real coefficients having negative zeros.

\begin{theorem}\cite{AESW}
\label{Th_negative}Consider a polynomial $P(T)=a_{m}T^{m}+a_{m-1}%
T^{m-1}+\cdots+a_{1}T+a_{0}\in\mathbb{R}[T]$. Then $P$ has only real and
nonpositive zeros if and only if the sequence $a_{0},a_{1},\ldots
,a_{m},0,0,0,\ldots$ is totally positive, that is if and only if all the
minors of the infinite matrix%
\[%
\begin{array}
[c]{ccccc}%
a_{0} & 0 & 0 & 0 & \cdots\\
a_{1} & a_{0} & 0 & 0 & \cdots\\
a_{2} & a_{1} & a_{0} & 0 & \cdots\\
a_{3} & a_{2} & a_{1} & 0 & \cdots\\
\vdots & \ddots & \ddots & \ddots & \cdots
\end{array}
\]
are nonnegative.
\end{theorem}

\begin{proposition}
Any morphism $\tilde{f}$ defined on $\hat{T}_{\delta}$ which extends the
positive morphism $f$ belongs to $\Mult(\hat{T}_{\delta})^{+}$.
\end{proposition}

\begin{proof}
Let $\tilde{f}$ be a morphism extending $f$. Set $\varphi(T)=\tilde{f}%
(\Phi)(T)$ that is
\[
\varphi(T)=\prod_{\gamma\in\Pi_{\delta}}(T+\tilde{f}(e^{\gamma},1)).
\]
By using the same arguments as in the proof of Proposition \ref{Prop_DecFI(T)}%
, we obtain that the coefficients of $\varphi(T)$ are the
\[
\tilde{f}(e_{k}(e^{\gamma_{1}},\ldots,e^{\gamma_{N}}),k)\in\mathbb{C}%
,k=0,\ldots,N
\]
Now it follows from the Jacobi-Trudi determinantal expression of the Schur
functions that the minors of the matrix defined from the coefficients of
$\varphi(T)$ as in Theorem \ref{Th_negative} coincide with the complex numbers%
\[
\tilde{f}(\boldsymbol{s}_{\Lambda}(e^{\gamma_{1}},\ldots,e^{\gamma_{N}%
}),\left\vert \Lambda\right\vert ),\Lambda\in\mathcal{P}_{N}%
\]
where $\mathcal{P}_{N}$ is the set of partitions with at most $N$ parts and
$\boldsymbol{s}_{\Lambda}(e^{\gamma_{1}},\ldots,e^{\gamma_{N}})$ is the
plethysm of the Schur function $\boldsymbol{s}_{\lambda}$ in $N$ variables
$X_{1},\ldots,X_{N}$ by the Weyl character $s_{\delta}$. If we consider any
young symmetrizer $c_{\Lambda}$ of shape $\Lambda$ in $\mathbb{R}[S_{l}]$, the
group algebra of the symmetric group $S_{l}$ (see \cite{Ful}), the space
\[
c_{\Lambda}\cdot V(\delta)^{\otimes l}\text{ such that }l=\left\vert
\Lambda\right\vert
\]
has indeed the structure of a $G$-module and
\[
\boldsymbol{s}_{\Lambda}(e^{\gamma_{1}},\ldots,e^{\gamma_{N}})=\mathrm{char}%
\left(  c_{\Lambda}\cdot V(\delta)^{\otimes l}\right)  .
\]
This shows that $\boldsymbol{s}_{\Lambda}(e^{\gamma_{1}},\ldots,e^{\gamma_{N}%
})$ decomposes as a sum of characters in $\{s_{\lambda}\mid\lambda\in
\delta^{\otimes|\Lambda|}\}$ with nonnegative integer coefficients.\ In
particular, $(\boldsymbol{s}_{\Lambda}(e^{\gamma_{1}},\ldots,e^{\gamma_{N}%
}),|\Lambda|)$ belongs to $\hat{T}_{\delta}^{+}$ and therefore we get that
$\tilde{f}(\boldsymbol{s}_{\Lambda}(e^{\gamma_{1}},\ldots,e^{\gamma_{N}%
}),\left\vert \Lambda\right\vert )=f(\boldsymbol{s}_{\Lambda}(e^{\gamma_{1}%
},\ldots,e^{\gamma_{N}}),\left\vert \Lambda\right\vert )$ is real nonnegative
since $f$ is assumed nonnegative.\ By Theorem \ref{Th_negative} this shows
that $-\tilde{f}(e^{\gamma},1)$ is real nonpositive for any $\gamma\in
\Pi_{\delta}$. Finally we obtain that $\tilde{f}(e^{\gamma},1)$ is real nonnegative
for any $\gamma\in\Pi_{\delta}$ and thus $\tilde{f}$ takes real nonnegative values
on $\hat{T}_{\delta}$.
\end{proof}

\bigskip

\begin{proof}
[Proof of Proposition \ref{multiRelationWeightCharacter}]Let $f\in
\Mult(\hat{T}_{\delta}^{+})^{+}$. By Proposition \ref{Th_negative}, there
exists $\tilde{f}\in \Mult(\hat{T}_{\delta})^{+}$ such that $\tilde
{f}(s_{\lambda},n)=f(s_{\lambda},n)$ for $(s_{\lambda},n)\in\hat{T}_{\delta
}^{+}$. Let $w\in W$ and $(s_{\lambda},n)\in\hat{T}_{\delta}^{+}$: since
$w^{-1}$ yields a multiplicity preserving bijection on $\Pi_{\lambda}$,
\begin{align*}
(\tilde{f}\circ w)(s_{\lambda},n)=  &  \sum_{\gamma\in\Pi_{\lambda}}%
K_{\lambda,\gamma}\tilde{f}(e^{w(\gamma)},n)\\
=  &  \sum_{\gamma\in\Pi_{\lambda}}K_{\lambda,w^{-1}(\gamma)}\tilde
{f}(e^{\gamma},n)\\
=  &  \sum_{\gamma\in\Pi_{\lambda}}K_{\lambda,\gamma}\tilde{f}(e^{\gamma
},n)=\tilde{f}(s_{\lambda},n).
\end{align*}
Thus, for all $w\in W$, $(\tilde{f}\circ w)_{|\hat{T}_{\delta}^{+}}=\tilde
{f}_{|\hat{T}_{\delta}^{+}}$. Let $w\in W$ be such that $M_{\tilde{f}\circ
w}\in\Delta$ and set $g=\tilde{f}\circ w$. Then, $g$ is an element of
$\Mult(\hat{T}_{\delta})^{+}$ such that $\chi(g)=f$ and $M_{g}\in\Delta$.
\end{proof}

\bigskip

A straightforward application of Proposition \ref{algebraicDescription} and
Proposition \ref{isomosphismeMult} yields the following corollary:

\begin{corollary}
\label{firstExpression} There exists $\vec{t}\in[0,1]_{\delta}^{d}$ such that
$\mathbf{0}^{c}(\vec{t})$ is $\delta$-admissible and such that $f$ is equal to
\[
f(s_{\lambda},n)=\frac{S_{\lambda,n\delta}(\vec{t})}{(S_{\delta
}(\vec{t}))^{n}}.
\]
We denote this function by $f_{\vec{t}}$.
\end{corollary}

We denote by $\Psi^{+}:[0,1]_{\delta}^{d}\longrightarrow \Mult(\hat{T}_{\delta}%
^{+})^{+}$ the map sending $\vec{t}$ to the multiplicative map $f_{\vec{t}}$
of the latter corollary.

\subsection{Injectivity of the map $\Psi^{+}$}

It remains to show that the map $\Psi^{+}$ is injective.

\begin{lemma}
\label{lemma_encadre}Let $f\in \Mult(\hat{T}_{\delta}^{+})$. Let $\vec{t}\in
[0,1]_{\delta}^{d}$. For any $(\lambda,n)\in\hat{T}^{+}_{\delta}$ we have%
\[
1\leq\frac{S_{\lambda,n\delta}(\vec{t})}{\vec{t}^{n\delta-\lambda}}\leq
\dim(V(\lambda)).
\]

\end{lemma}

\begin{proof}
On the first hand,
\[
S_{\lambda,n\delta}(\vec{t})=\sum_{\gamma\in\Pi_{\lambda}}K_{\lambda,\gamma}%
\vec{t}^{n\delta-\gamma}\geq\vec{t}^{n\delta-\lambda},
\]
which yields $1\leq\frac{S_{\lambda,n\delta}(\vec{t})}{\vec{t}^{n\delta
}-\lambda}$. On the other hand, since $t_{i}\leq1$ for all $1\leq i\leq d$,
\[
S_{\lambda,n\delta}(\vec{t})=\sum_{\gamma\in\Pi_{\lambda}}K_{\lambda,\gamma}%
\vec{t}^{n\delta-\gamma}\leq\sum_{\gamma\in\Pi_{\lambda}}K_{\lambda,\gamma}%
\vec{t}^{n\delta-\lambda}\leq\dim V(\lambda)\vec{t}^{n\delta-\lambda},
\]
yielding the other inequality
\[
\frac{S_{\lambda,n\delta}(\vec{t})}{\vec{t}^{n\delta-\lambda}}\leq
\dim(V(\lambda)).
\]

\end{proof}

\begin{corollary}
\label{equalConstant} Let $t=(t_{1},\dots,t_{d}),\tau=(\tau_{1},\dots,\tau
_{d})$ be such that $\Psi^{+}(\vec{t})=\Psi^{+}(\vec{\tau})$. Then $S_{\delta
}(\vec{t})=S_{\delta}(\vec{\tau})$.
\end{corollary}

\begin{proof}
For all $n\geq1$, $(n\delta,n)\in\hat{T}_{\delta}^{+}$. Thus, by Lemma
\ref{lemma_encadre}, we have
\[
1\leq S_{n\delta,n\delta}(\vec{t})\leq\dim(V(n\delta))\text{ and }1\leq
S_{n\delta,n\delta}(\vec{\tau})\leq\dim(V(n\delta)).
\]
This yields
\[
\frac{1}{\dim(V(n\delta))}\leq\frac{S_{n\delta,n\delta}(\vec{t})}%
{S_{n\delta,n\delta}(\vec{\tau})}\leq\dim(V(n\delta)).
\]
But
\[
\frac{S_{n\delta,n\delta}(\vec{t})}{S_{n\delta,n\delta}(\vec{\tau})}%
=\frac{S_{\delta}(\vec{\tau})^{n}\Psi(t_{1},\dots,t_{d})(s_{n\delta}%
,n)}{S_{\delta}(\vec{t})^{n}\Psi(\tau_{1},\dots,\tau_{d})(s_{n\delta},n)}%
=\frac{S_{\delta}(\vec{\tau})^{n}}{S_{\delta}(\vec{t})^{n}},
\]
the last equality being due to the fact that $\Psi(t_{1},\dots,t_{d}%
)=\Psi(\tau_{1},\dots,\tau_{d})$. Therefore, we have the inequality
\[
\frac{1}{\dim(V(n\delta))}\leq\frac{S_{\delta}(\vec{\tau})^{n}}{S_{\delta
}(\vec{t})^{n}}\leq\dim(V(n\delta)).
\]
Since $\dim(V(n\delta))$ is polynomial in $n$, necessarily $S_{\delta}(\vec
{t})=S_{\delta}(\vec{\tau})$.
\end{proof}

The proof of the injectivity uses the combinatorics of Littelmann paths and we assume that the reader is familiar with this theory. We refer to \cite{Lit1} for an introduction to the operator $f_{\alpha},\alpha\in S$ which are used in the following proofs. We
recall that $B(\delta)$ denotes the set of Littelmann paths obtained from a
path $\pi_{0}$ in $\Delta$ of weight $\delta$. We introduce moreover the
following decomposition of a $\delta$-admissible subset $S^{\prime}\subset S$.

\begin{definition}
Let $S^{\prime}\subset S$ be $\delta$-admissible. A Dynkin subchain of type
$\alpha$ and length $r$ is a sequence $(\alpha_{1},\dots,\alpha_{r})$ of
simple roots in $S^{\prime}$ such that $\alpha_{1}=\alpha$, $\langle\alpha
_{r},\delta\rangle\not =0$ and $\langle\alpha_{i},\alpha_{i+1}\rangle\not =0$
for $1\leq i\leq r-1$. The depth $d_{S^{\prime},\delta}(\alpha)$ of
$\alpha$ relatively to $S^{\prime}$ is the smallest integer corresponding to
the length of a Dynkin subchain of type $\alpha$.
\end{definition}

Note that any simple root of a $\delta$-admissible subset admits at least one
Dynkin subchain, since it belongs to an indecomposable root system which is
not orthogonal to $\delta$.

\begin{lemma}
\label{existenceHighestWeightinit} Let $\lambda\in P^{+}$ and $\alpha\in S$ be
such that $\langle\lambda,\alpha\rangle\not =0$. Then there exists $(\mu
,n)\in\hat{T}^{+}_{\lambda}$ such that $n\lambda-\mu=\alpha$.
\end{lemma}

\begin{proof}
Suppose that $\langle\lambda,\alpha\rangle>0$, and thus $\lambda-\alpha\in
\Pi_{\lambda}$. We denote by $\pi_{0}$ the Littelmann path of $B(\lambda)$
with weight $\lambda$. Then, $wt(f_{\alpha}(\pi_{0}))=\lambda-\alpha$.
Moreover, since $f_{\alpha}(\pi_{0})=\pi_{0}-v\alpha$ with
$v:[0,1]\longrightarrow\lbrack0,1]$, we have
\[
\langle f_{\alpha}(\pi_{0})(t),\alpha\rangle=\langle\pi_{0}(t),\alpha
\rangle-v(t)\langle\alpha,\alpha\rangle\geq-\langle\alpha,\alpha\rangle,
\]
and for all simple root $\alpha^{\prime}\not =\alpha$ and $t\in\lbrack0,1]$,
we have
\[
\langle f_{\alpha}(\pi_{0})(t),\alpha^{\prime}\rangle=\langle\pi_{0}%
(t),\alpha^{\prime}\rangle-v(t)\langle\alpha^{\prime},\alpha\rangle\geq0,
\]
because $\pi_{0}$ lies in the Weyl chamber $\Delta$ and $\langle\alpha
,\alpha^{\prime}\rangle\leq0$. Consider an integer $n\geq2$ such that
$\langle(n-1)\lambda,\alpha\rangle\geq\langle\alpha,\alpha\rangle$. Then, from
the two previous inequalities, $\pi_{0}^{\ast(n-1))}\ast f_{\alpha}(\pi_{0})$
lies in $\Delta$: thus, $wt\big(\pi_{0}^{\ast(n-1)}\ast f_{\alpha}(\pi
_{0})\big)=(n-1)\lambda+(\lambda-\alpha)$ is the highest weight of an irreducible
component of $V(\lambda)^{\otimes n}$, and $((n-1)\lambda+(\lambda
-\alpha),n)=(n\lambda-\alpha,n)\in\hat{T}_{\lambda}^{+}$: setting
$\mu=n\lambda-\alpha$ gives the result.
\end{proof}

\bigskip

The latter result can be generalized along a Dynkin subchain and yields the
following Lemma:

\begin{lemma}
\label{existenceHighestWeight} Let $S^{\prime}\subset S$ be $\delta
$-admissible and let $\alpha_{0}\in S^{\prime}$. There exists $(\lambda
,n)\in\hat{P}^{+}_{\delta}$ such that $n\delta-\lambda=\alpha_{0}%
+\sum_{\substack{\alpha^{\prime}\in S^{\prime}\\d(\alpha^{\prime})<
d(\alpha_{0})}}k_{\alpha^{\prime}}\alpha^{\prime}$.
\end{lemma}

\begin{proof}
Let $S^{\prime}\subset S$ be a $\delta$-admissible subset. We will prove the
result by induction on the depth of the simple root. For $d(\alpha)=1$, the
result is given by Lemma \ref{existenceHighestWeightinit}. Let
$i\geq2$. Suppose that the result is proven for all roots of depth at most
$i-1$, and let $\alpha$ be a root in $S^{\prime}$ of depth $i$. Let
$(\alpha,\alpha_{2},\dots,\alpha_{i})$ be a Dynkin chain of minimal length for
$\alpha$: by minimality, $\alpha_{j}$ has depth $i-j+1$ for $2\leq j\leq i$.
Since $d(\alpha_{2})=i-1$, there exists $(\lambda^{\prime},l)\in\hat
{T}_{\delta}^{+}$ such that $l\delta-\lambda^{\prime}=\alpha_{2}%
+\sum_{d(\alpha^{\prime})<d(\alpha_{2})}k_{\alpha^{\prime}}\alpha^{\prime}$.
If $\alpha^{\prime}$ is such that $d(\alpha^{\prime})<d(\alpha_{2})$, then
necessarily $\langle\alpha,\alpha^{\prime}\rangle=0$ (otherwise, there would
exist a Dynkin subchain of type $\alpha$ and length smaller than $i$);
likewise, since $d(\alpha)\geq2$, $\langle\delta,\alpha\rangle=0$. Thus,
\[
\langle\lambda^{\prime},\alpha\rangle=\left\langle l\delta-\alpha_{2}%
-\sum_{d(\alpha^{\prime})<d(\alpha_{2})}k_{\alpha^{\prime}}\alpha^{\prime
},\alpha\right\rangle =-\langle\alpha_{2},\alpha\rangle>0.
\]
Let $\pi$ be a Littelmann path in $B(\delta)^{\otimes l}$ lying in $\Delta$
and having weight $\lambda^{\prime}$. We consider $\pi$ as the Littelmann path
of highest weight for the irreducible representation $V(\lambda^{\prime})$.
Since $\langle\alpha,\lambda^{\prime}\rangle>0$, $\lambda^{\prime}-\alpha$ is
a weight of $V(\lambda^{\prime})$. Applying Lemma
\ref{existenceHighestWeightinit} yields the existence of $m\geq1$ such that
$\pi^{\ast m}\ast f_{\alpha}(\pi)$ lies in the Weyl chamber. Thus, $\pi^{\ast
m}\ast f_{\alpha}(\pi)$ correspond to a highest weight vector in
$(V(\delta)^{\otimes l})^{\otimes m}$. On the other hand,
\[
wt(\pi^{\ast m}\ast f_{\alpha}(\pi))=m\lambda^{\prime}-\alpha=lm\delta
-\alpha-m\alpha_{2}-m\sum_{d(\alpha^{\prime})<d(\alpha_{2})}k_{\alpha^{\prime}%
}\alpha^{\prime}=lm\delta-\alpha-\sum_{d(\alpha^{\prime})<d(\alpha)}%
k_{\alpha^{\prime}}^{\prime}\alpha^{\prime},
\]
with $k_{\alpha}^{\prime}\geq0$. Setting $\lambda=lm\delta-\alpha
-\sum_{d(\alpha^{\prime})<d(\alpha)}k_{\alpha^{\prime}}^{\prime}\alpha
^{\prime}$ and $n=lm$, we get an element $(\lambda,n)\in\hat{T_{\delta}}^{+}$
satisfying the hypothesis of the Lemma.
\end{proof}

\begin{corollary}
\label{differentSubset} Let $(t_{1},\dots,t_{d}),(\tau_{1},\dots,\tau_{d}%
)\in\lbrack0,1]_{\delta}^{d}$ and $1\leq i\leq d$ be such that $t_{i}=0$ and
$\tau_{i}\not =0$. Then, $\Psi(t_{1},\dots,t_{d})\not =\Psi(\tau_{1}%
,\dots,\tau_{d})$.
\end{corollary}

\begin{proof}
Note that $\mathbf{0}_{\vec{t}}^{c}$ and $\mathbf{0}_{\vec{\tau}}^{c}$ are
$\delta$-admissible subsets by definition of $[0,1]_{\delta}^{d}$. Since $\tau
_{i}\not =0$, $i\in\mathbf{0}_{\vec{\tau}}^{c}$. Thus, by Lemma
\ref{existenceHighestWeight}, there exists $(\lambda,n)\in\hat{T}_{\delta}%
^{+}$ such that $\lambda=n\delta-\alpha_{i}-\sum_{\substack{j\in \mathbf{0}_{\vec{\tau}}^{c}\\d(\alpha_{j})<d(\alpha_{i})}}k_{\alpha_{j}}\alpha_{j}$. Since $\tau_{j}>0$ for all
$j\in\mathbf{0}_{\tau}^{c}$,
\[
\Psi(\tau_{1},\dots,\tau_{d})(s_{\lambda},n)=\frac{S_{\lambda,n\delta}%
(\vec{\tau})}{S_{\delta}(\vec{\tau})^{n}}\geq\frac{\vec{\tau}^{n\delta
-\lambda}}{S_{\delta}(\vec{\tau})^{n}}=\frac{1}{S_{\delta}(\vec{\tau})^{n}%
}\tau_{i}\prod_{\substack{j\in\mathbf{0}_{\vec{\tau}}^{c}\\d(\alpha
_{j})<d(\alpha_{i})}}\tau_{j}^{k_{\alpha_{j}}}>0.
\]
On the other hand, any weight of $V(\lambda)$ has the form $\lambda
-\sum_{\alpha\in S}r_{\alpha}\alpha$ for some integer coefficients $r_{\alpha
}\geq0$; thus, since $t_{i}=0$, for any weight $\mu=\lambda-\sum_{\alpha\in
S}r_{\alpha}\alpha$ of $V(\lambda)$ we have
\[
\vec{t}^{n\delta-\mu}=t_{i}\prod_{\substack{j\in\mathbf{0}_{\vec{\tau}}%
^{c}\\d(\alpha_{j})<d(\alpha_{i})}}t_{j}^{k_{\alpha_{j}}}\prod_{\alpha_{j}\in
S}t_{j}^{r_{\alpha_{j}}}=0.
\]
Thus, $\Psi(t_{1},\dots,t_{d})(s_{\lambda},n)=0\not =\Psi(\tau_{1},\dots
,\tau_{d})$. This yields that $\Psi(t_{1},\dots,t_{d})\not =\Psi(\tau
_{1},\dots,\tau_{d})$.
\end{proof}

\begin{proposition}
\label{injectivity} The map $\Psi^{+}$ is injective.
\end{proposition}

\begin{proof}
Let $(t_{1},\dots,t_{d}),(\tau_{1},\dots,\tau_{d})\in\lbrack0,1]_{\delta}^{d}$ be
such that $\Psi(t_{1},\dots,t_{d})=\Psi(\tau_{1},\dots,\tau_{d})$. In this
case, Corollary \ref{equalConstant} yields that $S_{\delta}(\vec{\tau
})=S_{\delta}(\vec{t})$. By Corollary \ref{differentSubset}, we can assume
that $\mathbf{0}_{\vec{t}}^{c}=\mathbf{0}_{\vec{\tau}}^{c}$, and we will
denote this set $S^{\prime}$: we recall that the set of simple roots is
identified with $\{1,\dots,d\}$, so that $S^{\prime}$ corresponds to a $\delta
$-admissible subset of $S$. We will prove by induction on the depth of the
simple root $\alpha_{j}$ with respect to $S^{\prime}$ that $t_{j}=\tau_{j}%
$. Suppose that $\alpha_{j}\in S^{\prime}$ is such that $d(\alpha
_{j})=1$. By Lemma \ref{existenceHighestWeight}, there exists $n\geq1$ such
that $(n\delta-\alpha_{j},n)\in\hat{T}_{\delta}^{+}$. Thus, $(nk\delta
-k\alpha_{j},kn)\in\hat{T}_{\delta}^{+}$ for all $k\geq1$. Since $\Psi
(t_{1},\dots,t_{d})=\Psi(\tau_{1},\dots,\tau_{d})$, we have
\[
\frac{1}{S_{\delta}(\vec{t})^{kn}}S_{kn\delta-k\alpha_{j},kn\delta}(\vec
{t})=\frac{1}{S_{\delta}(\vec{\tau})^{kn}}S_{kn\delta-k\alpha_{j},kn\delta
}(\vec{\tau}),
\]
which simplifies into $S_{kn\delta-k\alpha_{j},kn\delta}(\vec{t}%
)=S_{kn\delta-k\alpha_{j},kn\delta}(\vec{\tau})$ because $S_{\delta}(\vec
{t})=S_{\delta}(\vec{\tau})$. By Lemma \ref{lemma_encadre}, we have
\[
1\leq\frac{S_{kn\delta-k\alpha_{j},kn\delta}(\vec{t})}{t_{j}^{k}}\leq\dim
V(kn\delta-k\alpha_{j})\text{ and }1\leq\frac{S_{kn\delta-k\alpha_{j}%
,kn\delta}(\vec{\tau})}{\tau_{j}^{k}}\leq\dim V(kn\delta-k\alpha_{j}).
\]
Thus,
\[
\frac{1}{\dim V(kn\delta-k\alpha_{j})}\leq\frac{t_{j}^{k}}{\tau_{j}^{k}}%
\leq\dim V(kn\delta-k\alpha_{j}).
\]
Since $\dim V(kn\delta-k\alpha_{j})$ is polynomial in $k$, necessarily
$t_{j}=\tau_{j}$. Let $i\geq2$, and suppose that we have proven that
$t_{j}=\tau_{j}$ for all $j$ such that $d(\alpha_{j})<i$. Let $\alpha_{l}$ be
such that $d(\alpha_{l})=i$. By Lemma \ref{existenceHighestWeight}, there
exists $(\lambda,n)\in\hat{T}_{\delta}^{+}$ such that $n\delta-\lambda
=\alpha_{l}+\sum_{\substack{\alpha^{\prime}\in S^{\prime}\\d(\alpha^{\prime
})<d(\alpha)}}k_{\alpha^{\prime}}\alpha^{\prime}$, with $k_{\alpha^{\prime}%
}\geq0$. Thus, for all $k\geq1$, $(k\lambda,kn)\in\hat{T}_{\delta}^{+}$. As in
the initial case, this implies that
\[
S_{k\lambda,kn\delta}(\vec{t})=S_{k\lambda,kn\delta}(\vec{\tau}),
\]
yielding together with Lemma \ref{lemma_encadre} the inequality
\begin{equation}
\frac{1}{dimV(k\lambda)}\leq\frac{\vec{t}^{nk\delta-\lambda}}{\vec{\tau
}^{nk\delta-\lambda}}\leq\dim V(k\lambda). \label{inequalitySumWeight}%
\end{equation}
But $nk\delta-k\lambda=k\alpha_{l}+k\sum_{\substack{j\in\mathbf{0}_{\vec{t}%
}^{c}\\d(\alpha_{j})<d(\alpha_{l})}}k_{\alpha_{j}}\alpha_{j}$, and by the
induction hypothesis, $t_{j}=\tau_{j}$ for all $j\in\mathbf{0}_{\vec{t}}%
^{c},d(\alpha_{j})<d(\alpha_{l})$. Thus
\[
\frac{\vec{t}^{nk\delta-\lambda}}{\vec{\tau}^{nk\delta-\lambda}}=\frac
{t_{l}^{k}}{\tau_{l}^{k}}.
\]
Since $\dim V(k\lambda)$ is polynomial in $k$, \eqref{inequalitySumWeight}
yields that $\frac{t_{l}}{\tau_{l}}=1$. This concludes the proof of
Proposition \ref{injectivity}.
\end{proof}

\begin{corollary}
The map $\Psi^{+}$ is a bijection from $[0,1]_{\delta}^{d}$ to $\Mult(\hat
{T}_{\delta}^{+})^{+}$. In particular, $\partial\mathcal{H}_{\infty}%
(\Delta)$ is isomorphic to $[0,1]_{\delta}^{d}$.
\end{corollary}

\section{Drift of a path following a central measure}\label{SectionMean}

In this section, we identify the set $\{(\vec{t},w)\in[0,1]_{\delta
}^{d}\times W|w\in W^{\mathbf{1}(\vec{t})}\}$ with $K(\delta)$ in order to
complete the proof of Theorem \ref{mainresult}: this identification is done by
considering the mean vector of the random walk given by the map $\Psi$. At the end of this section we prove Corollary \ref{corollary}.

\subsection{The mean vector $\vec{M}$}

Let us introduce the map
\[
\vec{M}:\left\{
\begin{matrix}
\{(\vec{t},w)\in[0,1]_{\delta}^{d}\times W|w\in W^{\mathbf{1}(\vec{t})}\} &
\longrightarrow & K(\delta)\\
\vec{t}\times w & \mapsto & M_{\Psi(\vec{t},w)}%
\end{matrix}
\right.  ,
\]
where the mean vector $M_{f}$ has been introduced in Section $3.2$ for any
multiplicative map $f\in \Mult(\hat{T}_{\delta})^{+}$. For $I\subset
\{1,\dots,d\}$, denote by $W_{I}$ the parabolic subgroup generated by the
simple roots $\alpha_{i}$ for $i\in I$.

\begin{lemma}
\label{positionDrift} Let $(\vec{t},w)\in\lbrace(\vec{t},w)\in[0,1]_{\delta
}^{d}\times W\vert w\in W^{\mathbf{1}(\vec{t})}\rbrace$. Then $M_{\Psi(\vec{t}%
,w)}\in w^{\prime-1}(\Delta)$ if and only if $w^{\prime}\in W_{\mathbf{1}%
(\vec{t})}w$.
\end{lemma}

\begin{proof}
Let $\alpha_{i}\in S$. We have
\[
M_{\Psi(\vec{t},w)}=\frac{1}{S_{\delta}(\vec{t})}\sum_{\gamma\in\Pi_{\delta}%
}K_{\delta,\gamma}\vec{t}^{\delta-w(\gamma)}\gamma.
\]
Thus,
\begin{align*}
\langle M_{\Psi(\vec{t},w)},w^{-1}(\alpha_{i})\rangle=  &  \frac{1}{S_{\delta
}(\vec{t})}\sum_{\gamma\in\Pi_{\delta}}K_{\delta,\gamma}\vec{t}^{\delta
-w(\gamma)}\langle\gamma,w^{-1}(\alpha_{i})\rangle=\frac{1}{S_{\delta}(\vec
{t})}\sum_{\gamma\in\Pi_{\delta}}K_{\delta,w(\gamma)}\vec{t}^{\delta
-w(\gamma)}\langle w(\gamma),\alpha_{i}\rangle\\
=  &  \frac{1}{S_{\delta}(\vec{t})}\sum_{\gamma\in\Pi_{\delta}}K_{\delta
,\gamma}\vec{t}^{\delta-\gamma}\langle\gamma,\alpha_{i}\rangle=\frac
{1}{S_{\delta}(\vec{t})}\sum_{\substack{\gamma\in\Pi_{\delta}\\\langle
\gamma,\alpha_{i}\rangle>0}}K_{\delta,\gamma}\vec{t}^{\delta-\gamma}(1-t_{i}^{2\frac
{\langle\gamma,\alpha_{i}\rangle}{\langle\alpha_{i},\alpha_{i}\rangle}%
})\langle\gamma,\alpha_{i}\rangle.
\end{align*}
Since each $t_{i}$ is in $\lbrack0,1]$, $\langle M_{\Psi(\vec{t},w)},w^{-1}(\alpha
_{i})\rangle\geq0$ for $1\leq i\leq d$ and hence $w(M_{\Psi(\vec{t},w)})\in \Delta$. Moreover, $\langle M_{\Psi(\vec{t},w)},w^{-1}(\alpha
_{i})\rangle=0$ if and only if $t_{i}=1$. Therefore, $w(M_{\Psi(\vec{t}%
,w)})\in w^{\prime}(\Delta)$ if and only if $w^{\prime}$ is a product of
reflections $s_{\alpha_{i}}$ such that $t_{i}=1$. Applying $w^{-1}$ to the
latter result yields the proof of the Lemma.
\end{proof}

\begin{proposition}
\label{injectivityM} The map $\vec{M}$ is injective.
\end{proposition}

\begin{proof}
Let $(\vec{t},w)$ and $(\vec{t}^{\prime},w^{\prime})$ be two elements of
$\{(\vec{t},w)\in[0,1]_{\delta}^{d}\times W|w\in W^{\mathbf{1}(\vec{t})}\}$
such that $\vec{M}(\vec{t},w)=\vec{M}(\vec{t}^{\prime},w^{\prime})$. We simply
denote by $M$ this common value. Lemma \ref{positionDrift} implies
that $W_{\mathbf{1}(\vec{t})}w=W_{\mathbf{1}(\vec{t}^{\prime})}w^{\prime}$.
Thus, $W_{\mathbf{1}(\vec{t})}=W_{\mathbf{1}(\vec{t}^{\prime})}$, which
implies that $\mathbf{1}(\vec{t})=\mathbf{1}(\vec{t}^{\prime})$; since $w$ and
$w^{\prime}$ are both a minimal right coset representative of $W_{\mathbf{1}%
(\vec{t}^{\prime})}w$, we have $w=w^{\prime}$. Let $F$ be the dominant
face corresponding to the $\delta$-admissible set $\mathbf{0}_{\vec{t}}^{c}$
and let $F^{\prime}$ be the one corresponding to the $\delta$-admissible set
$\mathbf{0}_{\vec{t}^{\prime}}^{c}$. By the results of Section \ref{Vinberg}, $\vec
{M}(\vec{t},w)\in w^{-1}(\overset{\circ}{F})$ (where $\overset{\circ}{F}$ the
interior of the face $F$) and $\vec{M}(\vec{t}^{\prime},w)\in w^{\prime
-1}(\overset{\circ}{F^{\prime}})$. Since $w=w^{\prime}$, we must have
$F=F^{\prime}$ and thus $\mathbf{0}_{\vec{t}}^{c}=\mathbf{0}_{\vec{t}^{\prime
}}^{c}$. Let $(X_{l})_{l\geq0}$, $(X_{l}^{\prime})_{l\geq0}$ be two
random walks with initial position $X_{0}=X_{0}^{\prime}=0$ and respective
transition matrices
\[
\mathbb{P}(X_{l+1}=\gamma|X_{l}=\gamma^{\prime})=K_{\delta,\gamma
-\gamma^{\prime}}\frac{\vec{t}^{\delta-w(\gamma-\gamma^{\prime})}}{S_{\delta
}(\vec{t})},\mathbb{P}(X_{l+1}^{\prime}=\gamma|X_{l}^{\prime}=\gamma^{\prime
})=K_{\delta,\gamma-\gamma^{\prime}}\frac{\vec{t}^{\prime\delta-w(\gamma
-\gamma^{\prime})}}{S_{\delta}(\vec{t}^{\prime})}.
\]
Both random walks have mean $M$, thus it follows by the local limit theorem
for large deviations (see for instance Theorem 4.2.1 in \cite{LLP1}) that for
any sequences of weights $(\gamma_{l})_{l\geq1},(\gamma_{l}^{\prime})_{l\geq
1}$ such that $\gamma_{l}-lM=o(l^{2/3}),\gamma_{l}^{\prime}-lM=o(l^{2/3})$,
and $\mathbb{P}(X_{l}=\gamma_{l})\not =0,\mathbb{P}(X_{l}=\gamma_{l}^{\prime
})\not =0$, we have
\begin{equation}
\mathbb{P}(X_{l}=\gamma_{l})\sim\mathbb{P}(X_{l}=\gamma_{l}^{\prime}),
\label{LLT}%
\end{equation}
and the same relation holds for $(X_{l}^{\prime})_{l\geq1}$. Let
$i\in\mathbf{0}_{\vec{t}}^{c}$. For $l\geq1$, let $(\gamma_{l},l)\in\hat{T}_{\delta}^{+}$ be such that $\gamma_{l}$ is an element
of $P_{\delta}\cap lF$ at minimal distance from $lM$ and set $\gamma
_{l}^{\prime}=\gamma_{l}-\alpha_{i}$. Then, $\mathbb{P}(X_{l}=\gamma
_{l})\not =0$. Since $M$ belongs to the interior of $M$, $\gamma_{l}^{\prime
}\in P_{\delta}\cap lF$ for $l$ large enough: thus, $\mathbb{P}(X_{l}%
=\gamma_{l}^{\prime})\not =0$ for $l$ large enough. The sequences $(\gamma
_{l}-lM)_{l\geq1}$ and $(\gamma_{l}^{\prime}-lM)_{l\geq1}$ are bounded, thus
the local limit Theorem applies and
\begin{equation}
\mathbb{P}(X_{l}=\gamma_{l})\sim\mathbb{P}(X_{l}=\gamma_{l}^{\prime})
\end{equation}
as $l$ goes to infinity. Since $X$ comes from a central measure,
\begin{equation}
\mathbb{P}(X_{l}=\gamma_{l})=\#\Gamma_{\mathbb{R}^{d}}(\gamma_{l},l)\frac{\vec{t}^{n\delta
-\gamma_{l}}}{\big(S_{\delta}(\vec{t})\big)^{l}}. \label{Gibbsnumberpaths}%
\end{equation}

Using \eqref{LLT} with \eqref{Gibbsnumberpaths} yields that $\frac
{\#\Gamma_{\mathbb{R}^{d}}(\gamma_{l},l)}{\#\Gamma_{\mathbb{R}^{d}}(\gamma'_{l},l)}\sim\vec{t}^{\gamma^{\prime
}_{l}-\gamma_{l}}=t_{i}^{-1}$. But the same holds for $X^{\prime}$, yielding
that $\frac{\#\Gamma_{\mathbb{R}^{d}}(\gamma_{l},l)}{\#\Gamma_{\mathbb{R}^{d}}(\gamma'_{l},l)}\sim\vec
{t}^{\prime\gamma^{\prime}_{l}-\gamma_{l}}=t^{\prime-1}_{i}$. Finally,
$t_{i}=t^{\prime}_{i}$.
\end{proof}

\bigskip

We can now prove the main result of this subsection:

\begin{proposition}
\label{bijectionPolytop} The map $\vec{M}$ is a bijective map from
$\lbrace(\vec{t},w)\in[0,1]_{\delta}^{d}\times W\vert w\in W^{\mathbf{1}(\vec{t}%
)}\rbrace$ to $K(\delta)$ such that $\vec{M}([0,1]_{\delta}^{d}\times
\Id)=K(\delta)^{+}$.
\end{proposition}

\begin{proof}
The injectivity of $\vec{M}$ has already been proven in Proposition
\ref{injectivityM}. Let us prove that $\vec{M}$ is surjective. Recall that
$\Psi$ is a restriction of the map $\Phi\circ\theta^{-1}:[0,1]^{d}\times
W\longrightarrow \Mult(\hat{T}_{\delta})^{+}$ defined at the end of Section
\ref{CharacMult}, and that both maps have the same image; thus, it is enough to prove
that the map $\vec{M}$ extended to the domain $[0,1]^{d}\times W$ by the
formula $\vec{M}(\vec{t},w)=M_{\Phi\circ(\theta^{-1}(\vec{t}),w)}$ is
surjective. Let us first prove that $\vec{M}_{|[0,1]^{d}\times \Id}$ is
surjective onto $K(\delta)^{+}$. Let $1\leq i\leq d$ be such that
$\langle\delta,\alpha_{i}\rangle\not =0$: then, $\alpha_{i}$ is a $\delta
$-admissible set, and the dominant face associated with $\alpha_{i}$ is
one-dimensional. Let $x_{i}=F_{i}\cap\partial\Delta$ (that is, $x_{i}$ is the
projection of $\delta$ on $\alpha_{i}^{\perp}$). Then, $K(\delta)^{+}$ is a
convex polytope whose extreme points are the elements $\delta,0$ and
$\{x_{i}\}_{\substack{1\leq i\leq d\\\langle\alpha_{i},\delta\rangle>0}%
}$. Let $\Sigma:(\mathbb{R}^{+})^{d}\longrightarrow\mathbb{R}$ be the
function defined by
\[
\Sigma(\vec{u})=\log(S_{\delta}(e^{u_{1}},\dots,e^{u_{d}}))=\log(\sum
_{\gamma\in\Pi_{\delta}}K_{\delta,\gamma}e^{\vec{u}.(\delta-\gamma)}).
\]
Then,
\[
\nabla\Sigma(\vec{u})=\left(  \frac{1}{S_{\delta}(e^{u_{1}},\dots,e^{u_{d}}%
)}\sum_{\gamma\in\Pi_{\delta}}K_{\delta,\gamma}(\delta_{i}-\gamma_{i}%
)e^{\vec{u}.(\delta-\gamma)}\right)  _{1\leq i\leq d}=\delta-\vec{M}%
((e^{u_{1}},\dots,e^{u_{d}}),\Id).
\]
Moreover, we can show that $\Sigma$ is a convex function: introduce the random
variable $X$ such that $\mathbb{P}(X=\delta-\gamma)=\frac{K_{\delta,\gamma}%
}{S_{\delta}(e^{u_{1}},\dots,e^{u_{d}})}e^{\vec{u}.(\delta-\gamma)}$. The
Hessian matrix of $\Sigma$ at $\vec{u}$ is exactly the covariance matrix of
the random variable $X$, which is nonnegative: since this is true for all
vector $\vec{u}\in(\mathbb{R}^{+})^{d}$, $\Sigma$ is indeed convex. 
Since $\Sigma$ is a convex function and $(\mathbb{R}^{+})^{d}$ is convex, the
set $\nabla\Sigma(\mathbb{R}^{+})^{d})$ is a convex set. We have thus proven
that the set $\{\delta-\vec{M}(e^{u_{1}},\dots,e^{u_{d}})|\vec{u}%
\in(\mathbb{R}^{+})^{d}\}=\vec{M}(]0,1]^{d},\Id)$ is convex. Note first that
$M(\mathbf{1},\Id)=0$, yielding that $0\in\vec{M}(]0,1]^{d}\times \Id)$. Let
$1\leq i\leq d$ be such that $\langle\alpha_{i},\delta\rangle\not =0$. Since
the map $S_{\delta}$ is continuous and nonzero on $[0,1]^{d}$, the map
$\vec{M}$ is continuous on $[0,1]^{d}\times \Id$: thus, if $(\vec{t}%
_{l})_{l\geq1}$ is a sequence of $]0,1]^{d}$ such that $\vec{t}_{l}%
\rightarrow(\delta_{ij})_{1\leq j\leq d}$ as $l$ goes to infinity, then
$\vec{M}(\vec{t}_{l},\Id)$ converges to $\vec{M}((\delta_{ij})_{1\leq j\leq
d},\Id)$. Let us set $(\alpha_{i},\delta):=\frac{2\langle\alpha
_{i},\delta\rangle}{\langle\alpha_{i},\alpha_{i}\rangle}$. For $1\leq l\leq(\delta,\alpha_{i})$, $K_{\delta,\delta
-l\alpha_{i}}=1$, since the only element of $B(\delta)$ ending at $\delta
-l\alpha_{i}$ is $f_{\alpha_{i}}^{l}(\pi_{0})$; thus, we have
\begin{align*}
\vec{M}((\delta_{ij})_{1\leq j\leq d},\Id)=  &  \frac{1}{S_{\delta}%
((\delta_{ij})_{1\leq j\leq d})}\sum_{l=0}^{(\alpha_{i},\delta)}%
K_{\delta,\delta-l\alpha_{i}}(\delta-l\alpha_{i})\\
=  &  \frac{1}{\sum_{l=0}^{(\alpha_{i},\delta)}K_{\delta,\delta-l\alpha_{i}}%
}\left(  \sum_{l=0}^{(\alpha_{i},\delta)}\delta-\sum_{l=0}^{(\alpha
_{i},\delta)}l\alpha_{i}\right) \\
=  &  \frac{1}{(\alpha_{i},\delta)+1}\left(  \big((\alpha_{i},\delta
)+1\big)\delta-\frac{(\alpha_{i},\delta).\big((\alpha_{i},\delta)+1\big)}{2}\alpha
_{i}\right) \\
=  &  \delta-\frac{(\alpha_{i},\delta)}{2}\alpha_{i}=x_{i},
\end{align*}
and $x_{i}$ belongs to $\overline{\vec{M}(]0,1]^{d}\times \Id)}$, the closure of
$\vec{M}(]0,1]^{d}\times \Id)$. Similarly, if $(\vec{t}_{l})$ is a sequence of
$[0,1]^{d}$ converging to $\mathbf{0}$, then $\vec{M}(\vec{t}_{l},\Id)$
converges to $\vec{M}(\vec{0},\Id)$. Since $\vec{M}(\vec{0},\Id)=\delta$,
$\delta\in\overline{\vec{M}(]0,1]^{d}\times \Id)}$. Hence, $0,\delta$ and
$\{x_{i}\}_{\substack{1\leq i\leq d\\\langle\alpha_{i},\delta\rangle>0}}$ are
in the closure of $\vec{M}(]0,1]^{d}\times \Id)$. Since $\vec{M}(]0,1]^{d}%
\times \Id)$ is convex, this yields that $K(\delta)^{+}\subset\overline{\vec
{M}(]0,1]^{d}\times \Id)}$. Since $[0,1]^{d}$ is compact, $\vec{M}%
([0,1]^{d},\Id)$ is compact and thus $\overline{\vec{M}(]0,1]^{d}\times
\Id)}\subset\vec{M}([0,1]^{d},\Id)$: this yields that $K(\delta)^{+}\subset\vec
{M}([0,1]^{d},\Id)$. By Lemma \ref{positionDrift}, $\vec{M}([0,1]^{d}%
,\Id)\subset K(\delta)^{+}$, so that finally $\vec{M}([0,1]^{d},\Id)=K(\delta
)^{+}$. Since $\vec{M}(\vec{t},w)=w^{-1}\vec{M}(\vec{t},\Id)$, $\vec
{M}([0,1]^{d}\times W)=\bigcup_{w\in W}w(K(\delta)^{+})=K(\delta)$.
\end{proof}

\subsection{Proof of Theorem \ref{mainresult}}

We give the proof of Theorem \ref{mainresult} by gathering the different
results of the paper. Let us prove the result only for $\partial
\mathcal{H}_{\infty}(\mathbb{R}^{d})$, since the proof for $\partial
\mathcal{H}_{\infty}(\Delta)$ is similar.

\begin{itemize}
\item By Corollary \ref{multiplicativeGraphDescription}, $\partial
\mathcal{H}_{\infty}(\mathbb{R}^{d})$ is homeomorphic to $\Mult(\hat{T}_{\delta})^{+}$ through
the map $j:\Mult(\hat{T}_{\delta})^{+}\longrightarrow\partial\mathcal{H}%
_{\infty}(\mathbb{R}^{d})$ defined by $j(f)(\Gamma(\tau))=f(\gamma,n)$ for
any path $\tau\in\Gamma_{\infty}(\mathbb{R}^{d})$ of length $n$ ending at $\gamma$.
Since $\Mult(\hat{T}_{\delta})^{+}$ is compact, $\partial\mathcal{H}_{\infty
}(\mathbb{R}^{d})$ is a compact space.

\item By Proposition \ref{isomosphismeMult} the map $\Psi:\{(\vec{t}%
,w)\in[0,1]_{\delta}^{d}\times W|w\in W^{\mathbf{1}(\vec{t})}%
\}\longrightarrow \Mult(\hat{T}_{\delta})^{+}$ given by $\Psi(\vec{t}%
,w)(\gamma,n)=\frac{1}{S_{\delta}(\vec{t})^{n}}\vec{t}^{n\delta-w(\gamma)}$ is
a bijection.

\item Finally, by Proposition \ref{bijectionPolytop}, the map $\vec{M}%
:\{(\vec{t},w)\in[0,1]_{\delta}^{d}\times W|w\in W^{\mathbf{1}(\vec{t}%
)}\}\longrightarrow K(\delta)$ given by $\vec{M}(\vec{t},w)=\frac{1}%
{S_{\delta}(\vec{t})}\sum_{\gamma\in\Pi_{\delta}}K_{\delta,\gamma}\vec
{t}^{\delta-w(\gamma)}$ is bijective.
\end{itemize}

Therefore, the map $\mathbb{P}:K(\delta)\longrightarrow\partial\mathcal{H}%
_{\infty}(\mathbb{R}^{d})$ given by $\mathbb{P}=j\circ\Psi\circ(\vec{M}%
^{-1})$ is a bijection. Note that from the previous results, for $m\in
K(\delta)$,
\[
\mathbb{P}_{m}(\Gamma(\tau))=\frac{1}{S_{\delta}(\vec{t}_{m})^{n}}\vec{t}%
_{m}^{n\delta-w(\gamma)},
\]
for all paths $\tau$ of length $n$ ending at $\gamma$. It remains to
show that $\mathbb{P}$ is indeed an homeomorphism. Since $K(\delta)$ and
$\partial\mathcal{H}_{\infty}(\mathbb{R}^{d})$ are compact, it suffices to
prove that $\mathbb{P}$ or $\mathbb{P}^{-1}$ is continuous. But for
$P\in\partial\mathcal{H}_{\infty}(\mathbb{R}^{d})$,
\[
\mathbb{P}^{-1}(P)=\sum_{\tau\in B(\delta)}P\big(\Gamma_{\mathbb{R}^{d}}(\tau)\big)\tau(1).
\]
Thus $\mathbb{P}^{-1}$ is continuous, which concludes the proof of Theorem
\ref{mainresult}. The same proof holds for $\partial\mathcal{H}_{\infty
}(\Delta)$ with $K(\delta)^{+}$ and the map $\mathbb{P}^{+}$ introduced in
the statement of the Theorem. For a metric space $X$, denote by
$M_{1}(X)$ the set of probability measures on $X$ with respect to its Borel
$\sigma$-algebra; we consider $M_{1}(X)$ as a topological space with the weak
convergence topology. As a straightforward corollary of Theorem
\ref{mainresult}, we get the following integral representation of
$\mathcal{H}_{\infty}(\mathbb{R}^{d})$ and $\mathcal{H}_{\infty}(\Delta)$.

\begin{corollary}\label{measureDescription}
The topological spaces $\mathcal{H}_{\infty}(\mathbb{R}^{d})$ and
$\mathcal{H}_{\infty}(\Delta)$ are homeomorphic to $M_{1}(K(\delta))$ and
$M_{1}(K(\delta)^{+})$, respectively through the maps
\[
\mathcal{P}:\left\{
\begin{matrix}
M_{1}(K(\delta)) & \longrightarrow & \mathcal{H}_{\infty}(\mathbb{R}^{d})\\
\mu & \mapsto & \int_{K(\delta)}\mathbb{P}_{m}d\mu(m)
\end{matrix}
\right.
\]
and
\[
\mathcal{P}:\left\{
\begin{matrix}
M_{1}(K(\delta)^{+}) & \longrightarrow & \mathcal{H}_{\infty}(\Delta)\\
\mu & \mapsto & \int_{K(\delta)^{+}}\mathbb{P}_{m}^{+}d\mu(m)
\end{matrix}
\right.  .
\]

\end{corollary}

We prove now that a random path in $\Gamma_{\infty}(\Delta)$ following the harmonic measure $\mathbb{P}^{+}_{m}$ admits a law of large numbers with drift $m$. In the case of a random path in $\Gamma_{\infty}(\mathbb{R}^{d})$ following the harmonic measure $\mathbb{P}_{m}$, the result is clear from the definition of $\mathbb{P}_{m}$ and the classical law of large numbers for random walks. The case of $\mathbb{P}^{+}_{m}$ is more complicated, since the random path is constrained to remain in a domain. However, the result is still true:
\begin{proposition}
Let $\gamma_{m}$ be a random path in $\Gamma_{\infty}(\Delta)$ following the harmonic measure $\mathbb{P}^{+}_{m}$. Denote by $\tau_{m}(n)$ the position of the path after $n$ steps. Then, almost surely,
$$\frac{1}{n}\tau_{m}(n)\longrightarrow m,$$
as $n$ goes to infinity.
\end{proposition}
\begin{proof}
Denote by $\tilde{\tau}_{m}$ the random path in $\Gamma_{\infty}(\mathbb{R}^{d})$ following the harmonic measure $\mathbb{P}_{m}$. By \cite[Theorem 4.12]{LLP3}, we have the equality in law
$$\tau_{m}=\mathcal{P}_{\alpha_{i_{1}}}\dots\mathcal{P}_{\alpha_{i_{r}}}(\tilde{\tau}_{m}),$$
where $w_{0}=s_{\alpha_{i_{1}}}\dots s_{\alpha_{i_{r}}}$ is a minimal length decomposition of the longest element of $W$, and each operator $\mathcal{P}_{\alpha}$ is the Pitman transformation associated with the root $\alpha$. We recall that the definition of the operator $\mathcal{P}_{\alpha}$ on a path $\tau\in\Gamma_{\infty}(\mathbb{R}^{d})$ is given by
$$\mathcal{P}_{\alpha}(\tau)(t)=\tau(t)-(\inf_{s\in[0,t]}\frac{2\langle \tau(s),\alpha\rangle}{\langle \alpha,\alpha\rangle})\alpha.$$
By a large deviation principle,
$$\Vert \frac{1}{t}(\tilde{\tau}_{m})_{\vert [0,t]}- m \Id_{\vert [0,t]}\Vert _{\infty}\underset{t\rightarrow+\infty}{\longrightarrow} 0$$
with probability one. Thus, for $s\in [0,t]$ and $\alpha\in S$,
$$ \left\vert\frac{1}{t}\frac{2\langle \tilde{\tau}_{m}(s),\alpha\rangle}{\langle \alpha,\alpha\rangle}-\frac{2\langle ms,\alpha\rangle}{\langle \alpha,\alpha\rangle}\right\vert\leq \epsilon(t),$$
with $\epsilon(t)$ converging to $0$ when $t$ goes to infinity with probability one. Since $m\in K(\delta)^{+}$, $\langle m,\alpha\rangle\geq 0$, and thus
$\inf_{s\in[0,t]}s\frac{2\langle m,\alpha\rangle}{\langle \alpha,\alpha\rangle}=0$.
Hence,
$$\left\vert\frac{1}{t}\inf_{s\in[0,t]}\frac{2\langle \tilde{\tau}_{m}(s),\alpha\rangle}{\langle \alpha,\alpha\rangle}\alpha\right\vert\leq \epsilon(t)\vert\alpha\vert\underset{t\rightarrow +\infty}\longrightarrow 0,$$
and finally $\frac{1}{t}\mathcal{P}_{\alpha}(\tilde{\tau}_{m})(t)\sim\frac{1}{t}\tilde{\tau}_{m}(t)\longrightarrow m$ as $t$ goes to $+\infty$, with probability one. Iterating this result for $\mathcal{P}_{\alpha_{i_{1}}},\dots,\mathcal{P}_{\alpha_{i_{r}}}$ yields that 
$$\frac{1}{t}\mathcal{P}(\tilde{\tau}_{m})(t)\underset{t\rightarrow +\infty}{\longrightarrow}m$$
with probability one. Since $\mathcal{P}(\tilde{\tau}_{m})$ is equal in law to $\tau_{m}$, the proof is done.
\end{proof}
\subsection{$c$-harmonic function killed on the boundary of $\Delta$}\label{proofCorollary}
We end this section by proving Corollary \ref{corollary}. We recall that $\hat{s}$ is the map from $\partial\mathcal{H}_{\infty}(\Delta)$ to $\mathbb{R}^{+}
\cup\lbrace\infty\rbrace$ defined by 
$$\hat{s}(\mathbb{P}^{+}_{m})=\sum_{\gamma\in P}K_{\delta,\gamma}\vec{t}_{m}^{\gamma}.$$
Note that the range of $\vec{t}_{m}$ is exactly $[0,1]_{\delta}^{d}$ by Proposition \ref{bijectionPolytop}. 
\begin{proof}[Proof of Corollary \ref{corollary}]

Suppose that $P\in\mathcal{H}_{c}(\Delta)$ and set $Z=\dim V(\delta)$. By Corollary \ref{measureDescription}, there exists $\mu\in \mathcal{M}_{1}(K(\delta)^{+})$ such that $P=\int_{K(\delta)^{+}}\mathbb{P}_{m}^{+}d\mu(m)$. Since $P\in\mathcal{H}_{c}(\Delta)$, there exists $p:\Lambda\cap \Delta\longrightarrow \mathbb{R}^{+}$ such that
$$P(\Gamma_{\Delta}(\gamma))=\frac{p(x)}{(cZ)^{n}},$$
for $\gamma\in\Gamma_{\Delta}(x,n)$. Let $\gamma=(\gamma_{n})_{n\geq 1}$ be a sequence of paths such that $\gamma_{n}\in \Gamma_{\Delta}(x,\phi(n))$ with $\phi(n)\longrightarrow +\infty$ when $n$ goes to infinity. Then, on the one hand, $(cZ)^{\phi(n)}P(\Gamma_{\Delta}(\gamma_{n}))$ is constant and equal to $p(x)$. On the other hand, by the expression of $P$ and Theorem \ref{mainresult},
\begin{equation}\label{fixedquantity}
(cZ)^{\phi(n)}P(\Gamma_{\Delta}(\gamma_{n}))=\int_{K(\delta)^{+}}(cZ)^{\phi(n)}\frac
{S_{x,\phi(n)\delta}(\vec{t}_{m})}{S_{\delta}(\vec{t}_{m})^{\phi(n)}}d\mu(m).
\end{equation}
Suppose that $m$ is such that $\hat{s}_{\delta}(\mathbb{P}_{m}^{+})=+\infty$. This means that there exists $1\leq i\leq d$ such that $(\vec{t}_{m})_{i}=0$. Thus, by Proposition \ref{descriptionNonZeroSet}, $\mathbb{P}_{m}^{+}(\Gamma_{\Delta}(\gamma))$ is nonzero only if $\gamma\in \Gamma_{\Delta}(y,n)$ for $y\in n\delta +n \Pi_{F}$, where $F$ is a fixed dominant face which is strictly smaller than $K_{\delta}$. Hence, for $n$ large enough, $\mathbb{P}_{m}^{+}(\Gamma_{\Delta}(\gamma_{n}))=0$, and we can assume that the support of $\mu$ is included in $\lbrace\mathbb{P}_{m}^{+}\vert\vec{t}_{m}\in ]0,1]^{d}\rbrace$. 
For $\vec{t}_{m}\in ]0,1]^{d}$, $s_{x}(\vec{t}_{m})$ is well defined for all $x\in P^{+}$ and we have
$$\frac
{S_{x,\phi(n)\delta}(\vec{t}_{m})}{S_{\delta}(\vec{t}_{m})^{\phi(n)}}=\frac{\vec{t}_{m}^{\phi(n)\delta}s_{x}(\vec{t}_{m})}{\vec{t}_{m}^{\phi(n)\delta}(s_{\delta}(\vec{t}_{m}))^{\phi(n)}}=\frac{s_{x}(\vec{t}_{m})}{\hat{s}_{\delta}(\mathbb{P}_{m}^{+})^{\phi(n)}}.$$
Thus, \eqref{fixedquantity} becomes
$$(cZ)^{\phi(n)}P(\Gamma_{\Delta}(\gamma_{n}))=\int_{K(\delta)^{+}}(cZ)^{\phi(n)}\frac
{s_{x}(\vec{t}_{m})}{\hat{s}_{\delta}(\mathbb{P}_{m}^{+})^{\phi(n)}}d\mu(m).$$
In order that the right hand-side of the latter expression does not go to infinity, we must have $\mu(\hat{s}_{\delta}^{-1}[0,cZ[)=0$. Then, we have 
$$p(x)=\lim_{n\rightarrow +\infty}\int_{K(\delta)^{+}}\left(\frac{tZ}{\hat{s}_{\delta}(\mathbb{P}_{m}^{+})}\right)^{\phi(n)}
s_{x}(\vec{t}_{m})d\mu(m)=\int_{\hat{s}_{\delta}^{-1}(\lbrace cZ\rbrace)}s_{x}(\vec{t}_{m})d\mu(m),$$
and the support of $\mu$ is $\hat{s}_{\delta}^{-1}(\lbrace cZ\rbrace)$. Finally, $\partial \mathcal{H}_{c}(\Delta)\subset \hat{s}_{\delta}^{-1}(\lbrace cZ\rbrace)$.\\
Its is readily seen that $\hat{s}_{\delta}^{-1}(\lbrace cZ\rbrace)\subset\mathcal{H}_{c}(\Delta)\cap \partial \mathcal{H}_{\infty}(\Delta)$ by the expression of $\mathbb{P}_{m}^{+}$ from Theorem \ref{mainresult}, and by the characterization of $c$-harmonic functions from Section \ref{doobconditioning}; thus $\partial \mathcal{H}_{c}(\Delta)=\hat{s}_{\delta}^{-1}(\lbrace cZ\rbrace)$, which proves the first part of the corollary.\\
For  the second part of the corollary, a quick computation yields that $\log(s_{\delta})$ is strictly convex on $\mathbb{R}^{d}$: the hessian matrix of $\log(s_{\delta})$ at $\vec{t}$ is actually the covariance matrix of a non-degenerate random variable. Thus, $\log(s_{\delta})$ admits a unique minimum on $\mathbb{R}^{d}$, which is located at the unique vector $\vec{t}_{0}$ such that $\nabla\log(s_{\delta})(\vec{t}_{0})={\vec{0}}$. Since 
$$\nabla\log(s_{\delta})(\vec{t})=\frac{1}{s_{\delta}(t)}\sum_{\gamma\in P}K_{\delta,\gamma}\vec{t}^{\gamma}\gamma=\vec{M}(\vec{t})$$
and $\vec{M}(\vec{1})=\vec{0}$, the minimum of $\log(s_{\delta})$ is at $\vec{1}$. The same holds for $s_{\delta}$, and thus $\min_{\partial_{\infty}(\Delta,0)}\hat{s}_{\delta}=\hat{s}_{\delta}(\mathbb{P}_{\vec{1}})=Z$. The second part of the corollary is a straightforward deduction of this fact.

\end{proof}
Note that the proof of the latter corollary gives an explicit expression of the harmonic measure associated with the central measure $\mathbb{P}_{m}$.
\bigskip

\noindent Laboratoire de Math\'{e}matiques et Physique Th\'{e}orique (UMR CNRS
7350).

\noindent Universit\'{e} Fran\c{c}ois-Rabelais, Tours Parc de Grandmont, 37200 Tours, France.

\noindent{cedric.lecouvey@lmpt.univ-tours.fr}
{pierre.tarrago@lmpt.univ-tours.fr}

\end{document}